\documentclass[11pt]{article}

\usepackage{mathtools}
\usepackage{amsmath, amsthm, amsfonts,amssymb}
\usepackage{enumitem}
\usepackage{graphicx}
\usepackage{colortbl}
\usepackage{tikz}
\usepackage[utf8]{inputenc}
\usepackage{esint}
\usepackage{mathrsfs}
\usepackage{subfig,float}
\usepackage[T1]{fontenc}
\usepackage{mathrsfs}  
\usepackage{bbm} 
\usepackage{enumitem}
\usepackage{mathtools}
\usepackage{dsfont}
\usepackage{ushort}
\usepackage{accents}
\newcommand{\ubar}[1]{\underaccent{\bar}{#1}}

\usepackage[english,french]{babel}

\def\d{\,\mathrm{d}}
\def\dv{\d v}

\newcommand{\sign}{\text{sgn}}
\DeclareMathOperator{\hess}{Hess}

\usepackage[colorlinks=true,linkcolor=blue,citecolor=blue,urlcolor=blue,breaklinks]{hyperref}

\usepackage{hyperref}
\usepackage{breakurl}
\usepackage[square,sort,comma,numbers]{natbib}
\usepackage{url}
\nocite{}

\usepackage[square,sort,comma,numbers]{natbib}
\definecolor{lpink}{rgb}{0.96, 0.76, 0.76}
\definecolor{dpink}{rgb}{0.97, 0.51, 0.47}
\definecolor{sky}{rgb}{0.53, 0.81, 0.92}
\definecolor{salmon}{rgb}{1.0, 0.55, 0.41}
\definecolor{orman}{rgb}{0.24, 0.7, 0.44}
\definecolor{aciksari}{rgb}{0.91, 0.84, 0.42}
\definecolor{dgrey}{rgb}{0.52, 0.52, 0.51}

\def\R{\mathbb{R}}

\def\d{\,\mathrm{d}}

\def\p{\,\partial}

\newtheorem{thm}{Theorem}[section]

\newtheorem{lem}[thm]{Lemma}
\newtheorem{prp}[thm]{Proposition}

\theoremstyle{definition}
\newtheorem{dfn}[thm]{Definition}
\theoremstyle{remark}
\newtheorem{remark}[thm]{Remark}

\hypersetup{pdftitle={RunAndTumbleHarris}}
\hypersetup{pdfauthor={Josephine Evans  \& Havva Yoldaş }}

\author{Josephine Evans\footnote{Warwick Mathematics Institute, University of Warwick, Zeeman Building, Coventry CV4 7AL, United Kingdom. josephine.evans@warwick.ac.uk} 
	\and Havva Yoldaş \footnote{Delft Institute of Applied Mathematics, Faculty of Electrical Engineering, Mathematics and Computer Science, Delft University of Technology, Mekelweg 4, 2628 CD Delft, Netherlands. h.yoldas@tudelft.nl}}

\title{On the asymptotic behaviour of a run and tumble equation for bacterial chemotaxis}

\makeindex

\begin{document}
	\selectlanguage{english}
	
	\maketitle

	\vspace{-10pt}
	
	\selectlanguage{english}
	\begin{abstract}
		\noindent 
		We prove that linear and \emph{weakly} non-linear run and tumble equations converge to a unique steady state solution with an exponential rate in a weighted total variation distance. In the linear setting, our result extends the previous results to an arbirtary dimension $d \geq 1$ while relaxing the assumptions. The main challenge is that even though the equation is a mass-preserving, Boltzmann-type kinetic-transport equation, the classical $L^2$ hypocoercivity methods, e.g., by Dolbeault, Mouhot, Schmeiser (Tans. Amer. Math. Soc., 367(6):3807-3828, 2015) are not applicable for dimension $d>1$. We overcome this difficulty by using a probabilistic technique, known as Harris's theorem.  We also introduce a \emph{weakly} non-linear model via a non-local coupling on the chemoattractant concentration. This toy model serves as an intermediate step between the linear model and the physically more relevant non-linear models. We build  a stationary solution for this equation and provide an hypocoercivity result.  
		
	\end{abstract}
	
	\tableofcontents
	
	\section{Introduction and main results} \label{sec:intro}
	
	We consider a kinetic-transport equation which describes the movement of biological microorganisms biased towards a chemoattractant. The model is called the \emph{run and tumble} equation and introduced in \cite {A80, S74} based on some experimental observations \cite{BB72} on the chemotaxis of the bacteria called \emph{E. coli} towards amino-acids. The equation is given by 
	\begin{align} \label{eq:rt}
		\partial_t f + v\cdot \nabla_x f = \int_{\mathcal{V}} \left( T(x, v, v') f(t,x,v') - T(x, v', v) f(t,x,v) \right) \d v', \quad t \geq 0,\,  x \in \R^d,  \, v \in \mathcal{V}.
	\end{align} where $f:=f (t,x,v) \geq 0$ is the density distribution of microorganisms at time $t \geq 0$ at a position $x \in \mathbb{R}^d$, moving with a velocity $v \in \mathcal{V} \subseteq \R^d$. In \eqref{eq:rt}, $\mathcal{V} =B(0, V_0)$ is a centered ball with a unit volume and a radius $V_0>0$ so that $|\mathcal{V} | =1$. Microorganisms perform a biased movement along the gradient of the chemoattractant with a constant speed and they change their orientation at random times towards the regions where the chemoattractant concentration is higher. This biased random walk drives the microorganisms up the gradient of the chemoattactant density. The underlying process is also called a \emph{velocity jump process}. 
	
	The tumbling frequency $T$ describes the change in velocity from $v$ to $v'$ and it can be written as
	\begin{equation} \label{defn:T}
		T(t, x, v, v'): =T( m, v, v') = \lambda (m) \kappa(v,v'), 
	\end{equation} where $\lambda : \mathbb{R} \to [0, \infty) $ is the tumbling rate and $m$ is the gradient of the external signal $M$ along the direction of $v'$ and given by
	\begin{align} \label{defn:m}
		m = v' \cdot \nabla_x M, 
	\end{align} where $M$ depends on the chemoattractant concentration $S$ via 
	\begin{align} \label{defn:logS}
		M= m_0 + \log S, 
	\end{align} where $m_0\in \R_0^+$ represents the external signal in the absence of a chemical stimulus. In \eqref{defn:T}, the tumbling kernel $\kappa$ is a probability distribution on $\mathcal{V}$ and gives the probability of  tumbling from velocity $v$ to velocity $v'$ so that it satisfies
	\begin{align*}
		\int_{\mathcal{V}} \kappa(v,v') \d v' = 1. 
	\end{align*}
	We assume that the distribution of the change in the velocity due to tumbling is uniform, i.e. $\kappa \equiv 1$. 
	If the chemoattractant density $S(x)$ is a fixed function of $x$, then  the equation becomes linear. Together with the above assumptions, the linear run and tumble equation takes the form
	\begin{align} \label{eq:rt-linear}
		\begin{split}
			\p_t f +v \cdot \nabla_x f &= \int _{\mathcal{V}}\lambda (v' \cdot \nabla_x M) f(t,x, v') \d v' - \lambda (v \cdot \nabla_x M) f(t,x,v), \\
			f(0,x,v) &= f_0(x,v), \qquad x \in \mathbb{R}^d, \, v \in \mathcal{V},
		\end{split} 
	\end{align} 
	where the initial datum $f_0$ is a probability measure, i.e., $f_0 \in \mathcal{P} (\R^d \times \mathcal{V})$\footnotemark \footnotetext{$\mathcal{P} (\Omega)$ denotes the space of probability measures defined on $\Omega$.}.
	
	Apart from the linear equation, we consider a model where the chemoattractant concentration $S$ solves 
	\begin{align} \label{eq:weakly_nonlin}
		S (t,x)= S_\infty (x)(1+ \eta (N * \rho)(t,x)), 
	\end{align}where $\eta>0$ is a small constant, $N$ is a positive, smooth function with a compact support, $S_{\infty}$ is a smooth function and $\rho(t, x) := \int_{\mathcal V} f(t, x,v)\d v$ is the spatial marginal density of microorganisms. 
	We refer to the problem \eqref{eq:rt-linear}-\eqref{eq:weakly_nonlin} as the \emph{weakly non-linear} run and tumble equation. This model can be considered as an intermediate model between the linear equation and physically relevant non-linear equations, e.g., when the microorganisms produce the chemoattractant by themselves $S$ solves a Poisson-type equation with a source term as the density,
	\begin{align} \label{eq:S-Poisson}
		- \Delta_x S + \alpha S = \rho, 
	\end{align} where $\alpha \geq 0$ is the chemical degradation rate. This non-linear model was first introduced in \cite{A80, ODA88} and further studied in \cite{CMPS04}. 
	
	The analytical results on the long-time behaviour of kinetic models of chemotaxis are scarce in the literature. We give a brief summary below. The main reason is that the classical hypocoercivity methods are not applicable to the run and tumble equation. We dedicate Section \ref{sec:motiv} to a detailed discussion on this matter. In the linear case, the closest result to ours is \cite{MW17} where the authors prove exponential convergence towards a non-trivial stationary state in $d\geq 1$ assuming that $S(x)$ is a radially symmetric function. In this paper, the techniques we use allow us to remove the radial symmetry assumption on $S$, thus we are able to generalise the result to $d\geq 1$, and a wider class of possible tumbling rates, with completely constructive arguments. However, these techniques cannot be utilised in the non-linear case. Therefore, we introduce a weakly non-linear equation to understand the link between the linear equation and the non-linear one \eqref{eq:rt-linear}-\eqref{eq:S-Poisson}. A detailed discussion of this connection and how we choose the coupling \eqref{eq:weakly_nonlin} can be found in Section \ref{sec:discussion}.
	
	\paragraph{Summary of previous results}  In this paper, we are concerned with the long-time behaviour of the run and tumble equation in the case that the solutions exist globally in time. Therefore, we do not provide an existence result. However, we would like to give a brief summary of the previous works, including the study of the Cauchy problem. We remark that the global existence of solutions for the models we study in this paper can be obtained by following the strategy e.g. in \cite{CMPS04, MW17} since the tumbling frequency we consider can be bounded by the necessary terms.

	The linear run and tumble model was studied in numerous works including \cite{OH00, OH02, CRS15, MW17}. 
	In \cite{CRS15}, the authors proved the existence and uniqueness of a non-trivial stationary state and exponential decay to equilibrium as $t \to \infty$ in  dimension $d=1$ by using modified entropy approach due to \cite{DMS15}. An example of a tumbling frequency satisfying the assumptions in \cite{CRS15} is given by 
	\begin{align}  \label{eq:T-sign}
		T(x,v,v') = 1 + \chi \sign (x \cdot v), \quad  \chi \in (0,1),
	\end{align} where $\chi$ is called the chemotactic sensitivity. Recently in \cite{MW17}, this result was extended to higher dimensions $d\geq1$ by considering splitting techniques due to \cite{MS16}. These techniques are based on using the Krein-Rutman theorem for positive semigroups which do not satisfy the necessary compactness assumptions. The general form of the tumbling frequency considered in \cite{MW17} is given by
	\begin{align} \label{eq:T-general}
		T  (x,v,v') = 1 -\chi \sign (\p_t S + v \cdot \nabla_x S), \quad \chi  \in  (0,1). 
	\end{align} In \cite{MW17}, the authors further assumed that the concentration of the chemoattractant $S(x)$ is radially symmetric and decreasing in $x$ such that $S(x) \to 0$ as $|x| \to \infty$. This assumption simplifies the tumbling kernel \eqref{eq:T-general} to \eqref{eq:T-sign} since the radial symmetry assumption reduces the problem essentially to dimension $d=1$. In this paper, we are able to remove the radial symmetry assumption and obtain the exponential convergence towards a unique stationary state in dimension $d\geq 1$. As in our case, when the concentration of the chemoattractant $S$ is a fixed function of $x$ but not necessarily radially symmetric or strictly decreasing in $|x|$, we refer to it as the linear problem. However we remark that, in \cite{MW17}, the authors refer to a specific case of the run and tumble equation as the linear problem. What we call the linear equation in this paper refers to more general form of the run and tumble equation.
	
	For the Cauchy problem, there are global existence results in \cite{CMPS04, HKS05, BCGP08, CP96}, and blow up results in \cite{BC09}. Moreover, in \cite{C19}, the author showed the existence of traveling wave solutions of a non-linear run and tumble model which is coupled with two reaction-diffusion equations. This analytical result complements the experimental observations and computational studies in \cite{SCBBSP10, SCBBSP11}. We refer also to \cite{BC10} for a detailed review of existence and blow-up results for various kinetic models of chemotaxis.

	\paragraph{Summary of our main results} In this paper, we prove that there exists a unique, non-trivial steady state solution to both the linear and the weakly non-linear equations and that both the models converge to the respective stationary states with explicitly computable exponential rates. In the linear case, our results are obtained by means of Harris's theorem which is a probabilistic method in ergodic theory of Markov processes. Moreover, we build a unique stationary solution for the weakly non-linear equation \eqref{eq:rt-linear}-\eqref{eq:weakly_nonlin} via a fixed-point argument and we show that the solutions coverge to this stationary solution exponentially using a perturbation argument. Indeed, $S(t,x)$ in \eqref{eq:weakly_nonlin} can be treated as a perturbation of the linear equation whenever $(N* \rho)(t, x)$ is decreasing or $\eta$ is small. The explicit rates of convergence can be obtained in terms of the constants given in the assumptions. Our proofs are all constructive and the estimates are in the weighted total variation distances, valid for arbitrary dimension, i.e. $ x \in \R^d, \,d\geq 1$.
	
	\subsection{Assumptions and main results} \label{sec:assumptions}
	
	We assume that the tumbling rate increases when the microorganisms move far away from the regions where the chemoattractant density is high and the chemoattractant density decreases as $|x| \to \infty$. We make the 
	following hypotheses: 
	\begin{itemize}
		\item[\bf(H1)] \label{hyp:lambda}
		The tumbling rate $\lambda(m): \mathbb{R} \to (0, \infty)$ is a function of the form
		\begin{align} \label{asm:lambda}
			\lambda(m) = 1- \chi\psi(m),  \quad \chi \in (0,1)
		\end{align} where $\psi$ is a bounded (with $\|\psi\|_\infty \leq 1$), odd, increasing function and $m\psi(m)$ is differentiable. 
		\item[\bf(H2)] \label{hyp:M}
		We suppose that  $M(x) \to-\infty$  as $|x| \to \infty$, $|\nabla_x M(x)|$ is bounded from above, and $\|\nabla_x M\|_\infty$ exists. Moreover, there exist $R\geq 0$ and $m_* >0$ such that whenever $|x|>R$ we have
		\[ |\nabla_x M(x)| \geq m_*. \]
		\item[\bf(H3)] \label{hyp:HessM}
		We suppose that $\hess (M)(x) \to 0$ as $|x|\to  \infty$ and  $|\hess(M) (x)|$ is bounded.
		\item[\bf(H4)] \label{hyp:int_bound}
		There exists a constant $\tilde{\lambda}>0$, depending on $\psi$ and $\|\nabla_x M\|_{\infty}$, and an integer $k>0$, depending on $\psi$, such that 
		\begin{align} \label{eq:int_bound}
			\int_{\mathcal{V}} 	\psi(v'\cdot \nabla_x M(x)) v'\cdot \nabla_x M(x)  \d v'   \geq \tilde{\lambda} (\psi, \|\nabla_x M\|_{\infty}) |\nabla_xM(x)|^k.
		\end{align} 
	\end{itemize}

\begin{remark}
We remark here that in addition to assumption \eqref{asm:lambda}, our theorem in the weakly non-linear setting also requires $\psi$ to be Lipschitz. This is so that we can ensure that a small difference in $\rho$ will result in a small difference in $\lambda(v\cdot \nabla_x M)$ when $M$ depends continuously on $\rho$. This type of assumption would in fact be necessary for the standard linearization of a coupled run and tumble equation to make sense and we believe it is a strength of our results that we are able to deal with $\psi$ other than the sign function.
\end{remark}
	
In order to explain where {\bf (H4)} comes from and justify its use we briefly prove \eqref{eq:int_bound} in two cases in the following lemma:
	\begin{lem}
		If $\psi(z) = \sign (z)$ then \eqref{eq:int_bound} holds with $k=1$ and 
		\begin{align*}
			\tilde{\lambda} = \int_{-V_0}^{V_0}|v_1|(V_0^2-v_1^2)^{(d-1)/2} \frac{\pi^{(d-1)/2}}{\Gamma ((d-1)/2 +1)}.
		\end{align*}If $\psi$ is differentiable with $\psi'(0)>0$ then \eqref{eq:int_bound} holds with $k=2$, and $\tilde{\lambda}$ depends on the exact form of $\psi$.
	\end{lem}
	\begin{proof}
		Since $\mathcal{V}$ is a ball of radius $V_0$, by rotation we obtain
		\begin{align*}
			\int_{\mathcal{V}} 	\psi(v'\cdot \nabla_x M(x)) v'\cdot \nabla_x M(x)  \d v'  = 	\int_{\mathcal{V}} \psi(v_1 |\nabla_x M(x)|) v_1 |\nabla_x M(x)| \mathds{1}_{\{v_2^2 + \dots + v_d^2 \leq V_0^2-v_1^2\}} \d v_1. 
		\end{align*} Integrating out $v_2, \dots, v_d$ gives
		\begin{align} \label{int1}
			\int_{-V_0}^{V_0}\psi (v_1 |\nabla_x M(x)|) v_1 |\nabla_x M(x)| (V_0^2-v_1^2)^{(d-1)/2} \frac{\pi^{(d-1)/2}}{\Gamma((d-1)/2 +1)} \d v_1.
		\end{align}
		We can bound \eqref{int1} below by
		\begin{align*} 
			\frac{\pi^{(d-1)/2}}{\Gamma((d-1)/2+2)} (V_0/2)^{d-1} \int_{-V_0/2}^{V_0/2} \psi(v_1|\nabla_x M(x)|) v_1 |\nabla_x M(x)| \d v_1. 
		\end{align*} From this point we extract the first result on $\psi(z) = \sign (z)$.

		For the case where $\psi$ is differentiable, we continue using the fact that $ \frac{\pi^{d/2}}{\Gamma(d/2 +1)} V_0^{d} = 1$ and changing variables from $v_1$ to $y = v_1|\nabla_xM|$, then the above bound is equal to
		\begin{align*}
			\frac{1}{2^{d-1}\sqrt{\pi}}\frac{\Gamma(d/2+1)}{\Gamma((d-1)/2+1)} \frac{1}{|\nabla_x M|V_0} \int_{-V_0 |\nabla_x M|/2}^{V_0|\nabla_x M|/2} \psi(y)y \d y. 
		\end{align*}
		Note that $\psi(y)y$ is a positive, even function which is $0$ at $y=0$. We have an average of $\psi(y)y$ over $-V_0|\nabla_xM|, V_0 |\nabla_x M|$ and it approaches to $0$ as $|\nabla_x M(x)| \to 0$. 
		Since $\psi$ is differentiable then $y\psi(y) \approx \psi'(0) y^2$ when $y$ is small so as $|\nabla_x M| \to 0$ we obtain
		\begin{align*}
			\frac{1}{2^{d-1}\sqrt{\pi}}\frac{\Gamma(d/2+1)}{\Gamma((d-1)/2+1)} \frac{1}{|\nabla_x M|V_0} \int_{-V_0 |\nabla_x M|/2}^{V_0|\nabla_x M|/2} \psi(y)y \d y	 \approx   \frac{1}{2^{d-1}\sqrt{\pi}}\frac{\Gamma(d/2+1)}{\Gamma((d-1)/2+1)}\psi'(0) \frac{V_0^2}{12}  |\nabla_x M|^2.
		\end{align*}
		This approximation only holds true as $|\nabla_x M|$ goes to 0, but since $|\nabla_x M|$ is a bounded function, and $\frac{1}{2^{d-1}\sqrt{\pi}}\frac{\Gamma(d/2+1)}{\Gamma((d-1)/2+1)} \frac{1}{|\nabla_x M|V_0} \int_{-V_0 |\nabla_x M|/2}^{V_0|\nabla_x M|/2} \psi(y)y \d y$ is a continuous function of $|\nabla_x M|$ and we have the result.
	\end{proof}
	
	We state the main results of the paper below. Their	proofs are given at the end of Sections \ref{sec:Harris} and \ref{sec:non-linear} respectively.
	
	\begin{thm}[The linear equation]
		\label{thm1}
		Suppose that $ t \mapsto f_t $ is the solution of Equation \eqref{eq:rt-linear} with initial data $f_0 \in \mathcal{P} (\R^d \times \mathcal{V})$. 
		We suppose that hypotheses \textbf{\em (H1)--(H4)} are satisfied. Then there exist positive constants $C, \sigma$ (independent of $f_0$) such that 
		\begin{align} \label{eq:thm1}
			\|f_t - f_{\infty} \|_* \leq C e^{-\sigma t} \| f_0 -f_{\infty}\|_*, 
		\end{align} where $f_{\infty}$ is the unique steady state solution to Equation \eqref{eq:rt-linear}. The norm $\|\cdot\|_*$ is the weighted total variation 
		defined by %, \havva{for any finite measure $\mu$ (or a measurable function)}, 
		\begin{align} \label{eq:norm1}
			\|\mu\|_*  := \int_{\R^d} \int_{\mathcal{V}}  \left( 1 -\gamma v\cdot \nabla_x M(x) -\beta \gamma \psi(v \cdot \nabla_x M(x)) v \cdot \nabla_x M(x)\right)e^{-\gamma M(x)}  |\mu| \d v \d x,
		\end{align} where $\gamma, \beta$ are positive constants which can be computed explicitly and are suffciently small so that the weight in $\|\cdot\|_*$ is positive.
		%	
		%	
		%	norm with a weight function
		%	\begin{align}
		%	\label{eq:norm1}
		%		\phi (x,v) =  e^{-\gamma M(x)}  \left( 1 -\gamma v\cdot \nabla_x M(x) -\beta \gamma \psi(v \cdot \nabla_x M(x)) v \cdot \nabla_x M(x)\right) 
		%	\end{align} where $\gamma, \beta$ are positive constants that can be computed explicitly.
		
		%defined by 
		%		\begin{align} 
		%			\|\mu\|_*  := \int_{\R^d} \int_{\mathcal{V}}  \left( 1 -\gamma v\cdot \nabla_x M(x) -\beta \gamma \psi(v \cdot \nabla_x M(x)) v \cdot \nabla_x M(x)\right)e^{-\gamma M(x)}  |\mu| \d v \d x,
		%		\end{align}  
		Furthermore, if there exist positive constants $ \ushort C,  \bar C,$ and $\alpha$ such that
		\begin{align*}
			\ushort C -\alpha \langle x \rangle \leq M (x) \leq \bar C- \alpha \langle x \rangle,
		\end{align*} then using equivalence of norms we can show a contraction as in $\eqref{eq:thm1}$ (with different constants $C$ and $\sigma$) in the weighted total variation norm with the weight $e^{\delta \langle x \rangle}$  
		\begin{align} \label{eq:norm2}
			\| \mu\|_{**} := \int_{\R^d} \int_{\mathcal{V}}  e^{\delta \langle x \rangle} |\mu| \d v \d x, 
		\end{align} 
		where $\delta$ is a small enough constant depending on $M$ and $\langle x \rangle = \sqrt{1+|x|^2}$.
	\end{thm}
\begin{remark}
We remark briefly on the Cauchy theory for this equation. Since it is linear we believe that the simplest way to see that the equation is well posed for initial data in $\mathcal{P}(\mathbb{R}^d \times \mathcal{V})$ is to directly construct a Markov process for Equation \eqref{eq:rt-linear}. We can do this by generating a series of jump times $J_1, J_2, \dots$ which are a Poisson process with rate $(1+\chi)$ and a series of post jump velocities $V_1, V_2, \dots$ which are independent and drawn from the uniform measure on $\mathcal{V}$ and lastly a series of thinning random variables $U_1, U_2, \dots$ which are i.i.d. drawn from the uniform distribution on $[0,1+\chi]$. Then we define a piecewise deterministic Markov process $(X_t, V_t)$ with $(X_0, V_0)$ having the prescribed distribution of the initial data and where if $t \in (J_i, J_{i+1})$ we have $X_t = X_{J_i} + (t-J_i)V_{J_i}$ and $V_t = V_{J_i}$ then we set $V_{J_{i+1}} = V_{i+1}$ if $U_{i+1} \leq \lambda(V_{J_i} \cdot \nabla_x M(X_{J_{i+1}})$ and $V_{J_{i+1}} = V_{J_i}$ otherwise.
\end{remark}
	\begin{thm} [The weakly non-linear equation]
		\label{thm2} 
		Suppose that $ t \mapsto F_t $ is the solution of Equation \eqref{eq:rt-linear} with the coupling \eqref{eq:weakly_nonlin} where we suppose that $N(x)$ is a positive, smooth function with a compact support, $\eta>0$ is a constant,  and $S_\infty(x)$ is a smooth function satisfying for some $\ubar  C, \bar C, \alpha >0$ that
		\begin{align*} 
			\ushort C -\alpha \langle x \rangle \leq M_\infty (x): =\log(S_\infty(x))\leq \bar C - \alpha \langle x \rangle, 
		\end{align*} where  $\langle x\rangle = \sqrt{1 + x^2}$. We suppose that hypotheses \textbf{\em (H1)--(H4)} are satisfied and that $\psi$ is a Lipschitz function. Then there exist some constant $\tilde{C}$ depending on $\ushort C, \bar C,$ and $\alpha$ such that if $\eta < \tilde{C}$ then there exists a unique steady state solution to Equation \eqref{eq:rt-linear}-\eqref{eq:weakly_nonlin}.
		Suppose further that, any initial data $F_0 \in \mathcal{P} (\R^d \times \mathcal{V})$ satisfying
		\begin{align*} 
			\|F_0\|_{**}  < \frac{1}{4} \left(\frac{\sigma^2}{4\eta \chi V_0 D \|\psi'\|_\infty \|\nabla_x N\|_\infty }-C^*\right),
		\end{align*} where $\sigma, D$ and $C^*$ are found in the proofs of Theorem \ref{thm1}, Proposition \ref{prop:unifeta} and Lemma \ref{lem:moment_weakly_non-linear} respectively, then we have that there exists a steady state of \eqref{eq:rt-linear} with the coupling \eqref{eq:weakly_nonlin} which we call $F_\infty$ and furthermore
		\begin{align*}
			\| F_t - F_\infty\|_{**} \leq Ce^{-\sigma t/2} \|F_0 -F_\infty\|_{**},
		\end{align*} where $C$ and $\sigma$ are some positive constants, and $\|\cdot\|_{**}$ is defined in \eqref{eq:norm2}.
	\end{thm}
\begin{remark}
A discussion of the Cauchy theory for the weakly non-linear equation can be found in the appendix.
\end{remark}

	\begin{remark} 
		Hypotheses \textbf{(H2)--(H4)} can be verified also in the case of the Poisson coupling \eqref{eq:S-Poisson}. The solution of $-\Delta W_y(x) + \alpha W_y(x) =\delta_x$ is called \emph{Yukawa potential} and given by the Green's function 
		\begin{align*}
			W_y(x) = \int_{0}^{\infty} \frac{1}{(4 \pi y)^{d/2} }\exp \left(- \frac{|x|^2}{4y}-\alpha y\right)\d y,
		\end{align*} and
		\begin{align*}
			-\frac{\log W_y(x)}{\sqrt{\alpha }|x|} \to 1 \text{  as  } |x| \to \infty.
		\end{align*} for dimension $d \geq 1$ (see \cite{LL01}, Theorem 6.23). We can see that $|\nabla_x M(x)|$ and $|\hess (M)(x)|$ are bounded and $\hess (M)(x) \to 0 $ as $|x| \to \infty$, where $M(x) =  \log W_y(x)$.
		
		Moreover the solution of $-\Delta_x S+ \alpha S=\rho $ is given by 
		\begin{align*}
			S := W_y * \rho = \int_{\R^d} W_y (x) \rho(t,y) \d y. 
		\end{align*} This case requires extra assumptions on $\rho$ in order to verify hypotheses \textbf{(H2)--(H4)}. 
		Since we do not deal with the Poisson coupling in this paper, we skip further details. 
	\end{remark}

	\paragraph{Structure of the paper} This paper is organised as follows. In Section \ref{sec:assumptions}, we listed the assumptions which are needed throughout the paper and presented the main results. We dedicated Section \ref{sec:motiv} to explaining the motivation, methodology and the novelty of our results. In Section \ref{sec:macroscopic}, we comment on the connection between the run and tumble equation and an aggregation-diffusion equation obtained as a parabolic scaling limit of the kinetic equation. In Section \ref{sec:Harris}, after stating Harris's theorem, in the subsequent two subsections we show how we verify two hypotheses of Harris's theorem for the linear run and tumble equation. We give the proof of Theorem \ref{thm1} at the end of Section \ref{sec:Harris}. Section \ref{sec:non-linear} is then devoted to the weakly non-linear case. In this section, we prove that there exists a unique stationary state solution and exponential convergence to this solution. Finally, Section \ref{sec:discussion} is dedicated to further discussions, particularly the connection between our results and the non-linear cases when different couplings for the chemoattractant density are considered.
	
	\subsection{Motivation, methodology and novelty}	\label{sec:motiv} 
	\paragraph{Motivation} 
	
	Our main motivation in this work is to simplify the proofs showing convergence to equilibrium for linear run and tumble equations, and extend their validity to a wider range of tumbling kernels and tumbling rates. We believe this moves the theory closer to being able to study the most biologically relevant tumbling rates and kernels, particularly existence and linear stability for the fully non-linear models. 
	
	Another motivation is that the equation is an interesting example within kinetic evolution equations. It differs from similar kinetic equations in a few key ways which we now describe. The linear equation \eqref{eq:rt-linear} has a structure similar to several equations appearing in the kinetic theory of gasses. In particular, we mention a \emph{linear Boltzmann equation} of the form
	\[ \partial_t f + v \cdot \nabla_x f - \nabla_x V(x) \cdot \nabla_v f = \left( \int_{\mathbb{R}^2} f(t,x,v') \d v'\right)\mathcal{M}(v) - f(t,x,v), \] where $f:= f (t,x,v)$ is the density distribution of particles at time $t$ in the phase space $(x,v)$, $V(x)$ is the confining potential, and $\mathcal{M}(v)$ is the Maxwellian velocity distribution. Long time behaviour for such equations is studied in the field of \emph{hypocoercivity}. We mention Villani's memoire \cite{V09} as the work which began the study of \emph{hypocoercivity} as a coherent behaviour common to many kinetic equations. The linear Boltzmann equation was first shown to converge to equilibrium by H\'erau in \cite{H07}. This also falls under the scope of the powerful general theorem in \cite{DMS15}. In \cite{CCEY20}, written by the authors and others, we show that Harris's theorem from Markov process theory provides an alternative way of showing convergence to equilibrium for the linear Boltzmann equation amongst other equations. 
	
	The run and tumble equation differs from the linear Boltzmann, and similar hypocoercive equations, in two key ways. Firstly, the confinement mechanism in the linear Boltzmann is through a `confining field' $\nabla_x V(x)$ whereas in the run and tumble equation the confinement is induced by the bias in the tumbling rate. This more complex confinement mechanism in the run and tumble equation is considerably more difficult to deal with. The second important difference between the linear Boltzmann equation and the run and tumble equation is the nature of the steady states. The steady states for the linear Boltzmann equation are simple and explicit and properties, such as Poincar\'e inequalities are immediate for such states. For the run and tumble equation, existence of a steady state is a problem in and of itself. The steady states for the run and tumble equation interact in a more complex way with the tools of hypocoercivity. For example, a condition in the theorem in \cite{DMS15} for proving a linear, mass-preserving kinetic equation is hypocoercive is that the steady state of the equation must be in the kernel of both the \emph{transport} and \emph{collision} operators separately. This is not possible for a steady state of the run and tumble equation, although we define the \emph{transport} ($v \cdot \nabla_ x $) and \emph{collision} (the right hand side of \eqref{eq:rt-linear}) parts of the operator. This behaviour is similar to non-equilibrium steady states in kinetic theory such as the ones discussed in \cite{AN00, EGKM13,CLM15,CELMM18, CELMM19}. Harris's theorem is well adapted to dealing with complex non-explicit steady states, and gives the existence of a steady state and the convergence to that steady state simultaneously. This fact was exploited by the first author in \cite{EM21} where we used Harris's theorem to find existence of a steady state for a non-linear kinetic equation with nonequilibrium steady states. Moreover, in \cite{CCEY20}, we showed that Harris's theorem can be applied efficiently to kinetic equations with nonlocal collision operators to obtain quantitative hypocoercivity results. In conclusion, the classical tools from hypocoercivity are difficult to apply on the run and tumble equation but Harris's approach gives promising results.
	
	Our motivation behind considering the weakly non-linear equation is to provide a useful intermediate step to treat the biologically more realistic couplings by means of exploring how a similar approach to ours in this paper can be applied to the fully non-linear case. This point is discussed in Section \ref{sec:discussion} in detail.
	
	\paragraph{Methodology} We obtain the spectral gap result in the linear case by applying Harris's thorem. In our case the Foster-Lyapunov condition which is necessary to use Harris's theorem is inspired from the moment estimates in \cite{MW17}. Using this type of argument to study asymptotic behaviour of biological models is a recent topic of research. One of the important recent results in this direction was \cite{G18} where the author used Doeblin's theorem, which is a predecessor of Harris's theorem, to obtain a spectral gap result for the renewal equation. In \cite{BCG20, BCGM19, CG20, CGY21, CY19}, Doeblin's and Harris's theorems were used for showing exponential contraction in weighted total variation distances for positive conservative and/or non-conservative semigroups, with several applications in population dynamics including the elapsed-time structured neuron population models, growth-diffusion and the growth-fragmentation equations. Particularly, in \cite{CY19} the authors improved previous results on the weakly non-linear model for interacting neuron dynamics. Their approach allows them to construct a steady solution to the non-linear equation based on an explicit smallness assumption on the connectivity parameter. The uniqueness of the stationary solution is then proved by a fixed point argument. The perturbation argument we used in the weakly non-linear setting is close to the ideas in \cite{CY19}. The main difference is that our fixed point argument is more involved and it requires the use of Harris's theorem and unlike in \cite{CY19} Doeblin's theorem does not work. Moreover, our argument requires additional moment estimates for the perturbation term. We carry out this by finding an appropriate Lyapunov functional in the non-linear case as well. This was not needed in \cite{CY19} as the authors could work with the steady solutions of the weakly non-linear equation explicitly.

	\paragraph{Novelty} The novelty in the present work is twofold. Firstly, to the best of our knowlegde, we give the most general spectral gap result (valid for arbitrary dimension $d\geq 1$ while relaxing the assumptions on the chemoattactant concentration $S$ and the tumbling frequency $T$) on the linear run and tumble equation. Particularly, our result is an improvement of the recent work in \cite{MW17}. Even though the result in \cite{MW17} is stated for $d \geq 1$, it is only valid under the assumption that the chemoattractant concentration $S$ is radially symmetric. Our result does not require this assumption to hold and it is valid for more general forms of tumbling frequency including the commonly used ones involving the ``sign function'' (see e.g. \eqref{eq:T-sign}, \eqref{eq:T-general}), in particular the fact that we study Lipschitz tumbling rates allows us to perform the later non-linear analysis and maybe helpful in dealing with the fully non-linear problem.
	
	Secondly, our results in the non-linear setting are all new. A non-local coupling \eqref{eq:weakly_nonlin} has not been considered in the literature before and there is not any explicit convergence result in the non-linear setting with any other type of non-linearity. We believe that our results on the weakly non-linear run and tumble equation are significant as they can be considered as an intermediate step towards studying the physically relevant case with Poisson coupling \eqref{eq:S-Poisson}.

	\subsection{Macroscopic models for chemotaxis}  \label{sec:macroscopic}
	
	Macroscopic models for chemotaxis are widely studied dating back to Patlak \cite{P53}, Keller and Segel \cite{KS70}.  Consequently, we describe briefly the relationship between kinetic and macroscopic models and the macroscopic models themselves. The motivation for this subsection is that the limiting aggregation-diffusion equation of the kinetic model that we study is an example of an more accurate macroscopic model for chemotaxis, the flux-limited Keller-Segel (FLKS) system.

	In \cite{KS70}, the authors study the aggregation behaviour of a population of a cell called \emph{D. discoideum} which performs \emph{amoeboid movement} by changing its shape to engulf bacteria or other substances like nutrients. The model describing this behaviour is referred to as the Patlak-Keller-Segel (PKS) system and given in the general form 
	\begin{align} \label{eq:KS1}
		\partial_t \rho &= \nabla \cdot \left( D_{\rho}\nabla \rho - \phi_S (\cdot) \rho \right), \\
		\partial_t S &= D_S \Delta S + g(\rho, S), \label{eq:KS2}
	\end{align} where $\rho:= \rho(t,x)$ is the cell density, $S:= S(t,x)$ is the chemoattractant concentration for $t \geq 0$, $x \in \R^2$ and  $D_{\rho} >0$, $D_S >0$ are the diffusivity of the cells and the chemoattractant respectively. The classical PKS system studied in \cite{KS70} corresponds to the case $\phi_S (\nabla S)= \chi \nabla S$, where the constant $\chi$ is the chemotactic sensitivity. In \eqref{eq:KS2}, $g$ is a function describing the production, degradation and consumption of the chemoattactant by the cells.  Typically, the cells move towards the regions with higher nutrient density. After consuming all the nutrient, they disperse uniformly over the space and, eventually, they start to aggregate and form clusters. The aggregation describes the instability observed in the population level and it is analogous with many physical problems. The significance of the PKS model comes from the fact that it allows to investigate aggregation behaviour of the population. There are numerous results linking the mesoscopic and microscopic descriptions of chemotaxis to the macroscopic one, e.g. \cite{A80, S00,DN94, OH00, OH02, PVW20, X15, JV13, STY14} and references therein.
	
	For the particular model we study here we can look at a limit to a particular macroscopic model following \cite{OH00, OH02}. We call $\tau$ and $\xi$ the scaled time and space variables respectively. For a small $\varepsilon >0$, $\tau$ and $\xi$ are given by 
	\begin{align*}
		\tau = \varepsilon^2 t,  \quad \xi = \varepsilon x.
	\end{align*} 
	We define $\lambda^\varepsilon(v, \xi):= \lambda (v \cdot \nabla_x M (\xi) )$ and assume that as $\varepsilon  \to 0$,
	\begin{align*} 
		\lambda^\varepsilon(v, \xi) \approx 1 - \varepsilon \chi \psi \left( v\cdot(\nabla_\xi M)(\xi) \right).
	\end{align*} This is consistent with the form of $\lambda$ we assumed in this paper (see hypothesis \textbf{(H1)}).
	
	We call $F(\tau, v, \xi )$ the density distribution of microorganisms with the scaled variables and we have the following equation for $F$,
	\begin{align*} 
		\varepsilon^2 \partial_\tau F + \varepsilon v \cdot \nabla_\xi F = \int_{\mathcal{V}} \lambda^\varepsilon(v',\xi)F(\tau, v',\xi)\d v' - \lambda^\varepsilon(v, \xi) F.  
	\end{align*}
	We define the new spatial density,
	\begin{align*}
		\rho(\tau, \xi) := \int F(\tau, v, \xi)\d v.
	\end{align*}
	We then have, by formal computations in the limit as $\varepsilon \to 0$,
	\begin{align} \label{eq:aggr-diff}
		\partial_\tau \rho = \nabla_\xi \cdot \left( \nabla_\xi \rho - \phi_S(\xi)\rho \right), 
	\end{align} where the macroscopic chemotactic velocity $\phi_S$ is given by
	\begin{align*}
		\phi_S  = \chi \int_{\mathcal{V}} v' \psi(v' \cdot (\nabla_\xi M)(\xi))\d v'.
	\end{align*}
	
	This model is slightly different and is an example of a flux-limited Keller-Segel equation which appears to be a more accurate description of chemotaxis by taking into account the saturation of the cell velocity. This model was introduced and studied in \cite{HP09, DS05,  CKWW12} amongst other work.

	\section{Harris's Theorem} \label{sec:Harris}
	
	In this section, we give the statement of Harris's theorem based on \cite{HM11, H16, CM21}. Harris's theorem is a probabilistic method which gives simple conditions on ergodic (long-time) behaviour of Markov processes. The original idea dates back to Doeblin \cite{D40} where he showed \emph{mixing }of a Markov chain whose transition probabilities possess a uniform lower bound. We refer to this condition as \emph{Doeblin condition} and explain it below. The mixing of a Markov chain refers to the time until the Markov chain reaches its stationary state distribution. In \cite{H56}, Harris studied the necessary conditions for a Markov process to admit a unique stationary state or an invariant measure. Later in \cite{ DMT95, MT93, MT09}, this result was used for the first time to obtain quantitative convergence rates based on verifying a \emph{minorisation condition} and a \emph{geometric drift} or \emph{Foster-Lyapunov condition}. In \cite{HM11}, the authors provided a simplified proof of Harris's theorem by using appropriate Kantorovich distances and recently in \cite{CM21}, the authors provided an alternative proof by using semigroup arguments. We state the theorems below in the spirit of \cite{H16, CM21}. 
	
	We consider a Polish space $\Omega$ and denote $\Sigma$ as the $\sigma-$algebra of Borel subsets of $\Omega$. Then $(\Omega, \Sigma)$ is a measurable space; and, endowed with any probability measure, $\Omega$ is a Lebesgue space. We denote the space of probability measures by $\mathcal{P}(\Omega)$.
	
	\,
	
	A natural way to construct a Markov process is via a \emph{transition probability function}.
	\begin{dfn}  \label{defn:trans_prob_func}
		A linear, measurable function $\mathcal{M}(x,A)$ is a transition probability function on $(\Omega, \Sigma)$ if for every $x$, $\mathcal{M}(x, \cdot)$ is a probability measure on $(\Omega, \Sigma)$  and $\mathcal{M}(\cdot, A)$ is a measurable function for every $A \in \Sigma$.
	\end{dfn}
	A Markov operator $M$ and its adjoint $M^*$ can be defined by means of a transition probability function $\mathcal{M}$ in the following way:
	\begin{align*}
		(M\mu) (A) = \int_{\Omega} \mathcal{M}(x,A) |\mu| \d x, \qquad (M^* \phi ) (x) = \int_{\Omega} \phi(y) \mathcal{M}(x, \d y),
	\end{align*} where $\phi : \Omega \mapsto [0, +\infty)$ is a bounded measurable function.
	\begin{dfn}
		A family of Markov operators $(M_{t})_{t \geq 0}$ is called a \emph{Markov semigroup} if it satisfies
		the following
		\begin{enumerate}
			\item [\textit{i.}] $M_0 =\mbox{Id}$ or equivalently $\mathcal{M}_0(x, \cdot) = \delta_x$ for all $x \in \Omega$.
			\item[\textit{ii}.] The semigroup property: $M_{t+s} = M_t M_s$ for $t,s \geq 0$.
			\item [\textit{iii}.] For every $\mu\in L^1$, $t \mapsto M_t\mu$ is continuous.
		\end{enumerate}
	\end{dfn}
	We also note that Markov semigroups have 
	\begin{enumerate}
		\item [\textit{i.}] Positivity property: $M_t \geq 0 $ for any $t \geq 0$
		\item[\textit{ii.}] Conservativity property:  $\int | M_t \mu| (\d x)= \int |\mu | (\d x)$ for any finite measure $\mu $. %f\in$. % \Omega$. 
		
		%$\langle M_t f \rangle =\langle f\rangle $ for any $f\in \Omega$ where $\langle  f \rangle : = \langle f, \textbf{1} \rangle$. 
	\end{enumerate}
	
	In our setting $M_t \mu$ will be the solution of the partial differential equation $f$  at time $t$ with an initial data $\mu$ which is a probability measure. Moreover for every $ t \geq 0$, if $M_t \mu = \mu$, then the probability measure $\mu$ is called an \emph{invariant measure} of $(M_t)_{t \geq 0}$ or equivalently a \emph{steady state solution} of $f$. 
	
	\setlength{\leftmargini}{15pt}
	
	\begin{thm}[Doeblin's Theorem] \label{thm:Doeblin}
		Suppose that we have a Markov semigroup $(M_t)_{t \geq 0}$
		which satisfies 
		
		\begin{itemize}
			\item [] \textbf{Doeblin's condition:} There exists a time $T>0$, a probability distribution $\nu$ and a constant $\alpha \in (0,1)$ such that for any $z_0$ in the domain
			\begin{align*} \label{eqn:Doeblin_condition}
				M_{T} \delta_{z_0} \geq \alpha \nu.
			\end{align*}
		\end{itemize} Then for any two finite measures $\mu_1$ and $\mu_2$ and any integer
		$n \geq 0$ we have that
		\begin{equation*} \label{eqn:Doeblin1}
			\left\| M^{n}_{T} (\mu_1 -  \mu_2)\right\|_{\mathrm{TV}} \leq (1-\alpha) ^n\left\| \mu_1 - \mu_2 \right\|_{\mathrm{TV}}.
		\end{equation*}
		As a consequence, the semigroup has a unique invariant probability measure $\mu_{\infty}$, and for all probability measures $\mu$
		\begin{equation*} \label{eqn:Doeblin2}
			\left\| M_{t} ( \mu - \mu_{\infty}) \right\|_{\mathrm{TV}}  \leq C e^{-\sigma t} \left\| \mu - \mu_{\infty} \right\|_{\mathrm{TV}}, \quad \text{ for all }t \geq 0,
		\end{equation*} where $C := 1/ (1-\alpha) > 1$ and $\sigma := -\log(1-\alpha)/T >0$.
	\end{thm}
	Doeblin's condition sometimes referred as the \emph{strong positivity condition} or \emph{uniform minorisation condition}. It means for a Markov process that the probability of transitioning from any initial state to any other state is positive. 
	Doeblin's theorem gives a unique stationary state for a Markov process and exponential convergence to this state once Doeblin's condition is satisfied. However, proving such a uniform positivity is often difficult. Especially when the state space of the Markov process is unbounded. Harris's theorem is an extension of Doeblin's theorem to these cases. Instead of a uniform minorisation condition, we show that Doeblin's condition is satisfied only in a given region and verify that the process will visit this region often enough. For the latter part we need to find an appropriate Lyapunov functional, i.e., verify the Foster-Lyapunov condition.
	
	\begin{thm}[Harris's Theorem] \label{thm:Harris} Suppose that we have a Markov semigroup $(M_t)_{t \geq 0}$ satisfying the following two conditions
		\begin{itemize}
			\item [] \textbf{Foster-Lyapunov condition:}
			There exists  $\lambda >0$, $K\geq 0$, some time $T> 0$ and a measurable function $\phi $ such that for all $z$ in the domain
			\begin{align} \label{con:Foster-Lyapunov}
				(M_T^* \phi) (z) \leq \lambda \phi (z) + K.
			\end{align}
			\item [] \textbf{Minorisation condition:}  There exists a time $T>0$, a probability distribution $\nu$ and a constant $\alpha \in (0,1)$ such that for any $z_0 \in \mathcal{C}$ ,
			\begin{align} \label{con:minorisation}
				M_{T} \delta_{z_0} \geq \alpha \nu,
			\end{align} where $\mathcal{C} : = \{z : \phi (z) \leq R\}$, for some  $R> 2K/ (1-\alpha)$.
		\end{itemize} Then there exist $\beta>0$ and $\bar{\alpha} \in (0,1)$ such that 
		\begin{align*}
			\left\| M^{n}_{T} (\mu_1 -  \mu_2)\right\|_{\phi, \beta} \leq \bar{\alpha} \left\| \mu_1 - \mu_2 \right\|_{\phi, \beta} 
		\end{align*} for all nonnegative measures $\int \mu_1 = \int \mu_2$ where the norm $\| \cdot \|_{\phi, \beta} $ is defined by
		\begin{align*}
			\|\mu\|_{\phi, \beta}  := \int (1 + \beta \phi (z)) |\mu| \d z.
		\end{align*}
		Moreover, the semigroup has a unique invariant probability measure $\mu_{\infty}$ and there exist  $C>1$, $\sigma >0$ (depending on $T, \alpha, \lambda , K, R$ and $\beta$) such that
		\begin{equation*}
			\left\| M_{t} ( \mu - \mu_{\infty}) \right\|_{\phi, \beta}   \leq C e^{-\sigma t} \left\| \mu - \mu_{\infty}\right\|_{\phi, \beta}, \quad \text{ for all }t \geq 0,
		\end{equation*}
	\end{thm}
	
	\begin{remark}
		The constants in Theorem \ref{thm:Harris} can be calculated explicitly. If we set $\lambda_0 \in [\lambda + 2K/R, 1)$ for any $\alpha_0 \in (0, \alpha )$ we can choose $\beta = \alpha_0/K$ and $\bar{\alpha} = \max \{1-\alpha - \alpha_0, (2+R\beta \lambda_0)/(2+ R\beta )\}$. Then we have $C:= 1/ \bar{\alpha}$ and $\sigma = - \log \bar{\alpha} / T$. 
	\end{remark} For the proofs of Theorem \ref{thm:Doeblin} and Theorem \ref{thm:Harris} we refer to \cite{HM11, H16} and references therein.
	
	In the following two sections we show how the Foster-Lyapunov condition and the minorisation condition are verified for Equation \eqref{eq:rt-linear}. At the end of the section we give the proof of Theorem \ref{thm1}.
	
	We use the notations $z := (x,v)$ and $\int  \d z := \int_{\R^d}\int_{\mathcal{V}} \d x \d v$ for the rest of the paper whenever convenient.
	
	\subsection{Foster-Lyapunov condition} \label{sec:Foster-Lyapunov}
	In this section, we verify the Foster-Lyapunov condition \eqref{con:Foster-Lyapunov} for Equation \eqref{eq:rt-linear}. In order to look at Lyapunov functions let us fix some notation. We remark that by \emph{Lyapunov functions} we do not refer to scalar functions which are used for stability results in ODE theory. By a Lyapunov function in the sense of Harris's theorem, we want some function $\phi(z)$ where $\phi(z) \rightarrow \infty$ as $|z| \rightarrow \infty$ and the existence of some $t>0$, $C>0$ and $\alpha \in (0,1)$ such that 
	\begin{align} \label{ineq:lyapunov1}
		\int \phi(z) f(t,z) \mathrm{d}z \leq \alpha \int \phi(z)  f_0(z) \d z + C \int f_0(z) \d z, 
	\end{align} for any initial data $f_0(z) \in \mathcal{P}(\R^d\times \mathcal{V})$. 
	
	For $f$ satisfying an equation
	\begin{align*}
		\partial_t f = \mathcal{L} f,
	\end{align*}we can prove \eqref{ineq:lyapunov1} by showing that
	\begin{align} \label{ineq:lyapunov2}
		\mathcal{L}^* \phi \leq - \gamma \phi + D, 
	\end{align} for some positive constants $\gamma, D$. One can verify the fact that \eqref{ineq:lyapunov2} implies \eqref{ineq:lyapunov1} with $\alpha= e^{-\gamma t} $ and $C = D/ \gamma$ by an easy computation (see, e.g., Remark 1 in \cite{Y22}).

	In \eqref{ineq:lyapunov2}, $\mathcal{L}^*$ is the formal adjoint of $\mathcal{L}$. In our case
	\begin{align}
		\label{eq:L_generator}
		\mathcal{L}f = - v\cdot \nabla_x f + \int_{\mathcal{V}} \lambda(v'\cdot \nabla_x M) f(x,v')\d v' - \lambda(v \cdot \nabla_x M) f(x,v). 
	\end{align} Therefore,
	\begin{align}
		\label{eq:L_adjoint}
		\mathcal{L}^*\phi = v\cdot \nabla_x \phi + \lambda(v\cdot \nabla_x M) \left( \int_{\mathcal{V}}\phi(x,v') \d v' - \phi(x,v)\right). 
	\end{align}
	
	\begin{lem}[Foster-Lyapunov condition for Equation \eqref{eq:rt-linear}] \label{lem:Foster-Lyapunov}
		Suppose that hypotheses \textbf{\em (H2)--(H4)} hold. Suppose also that $\|\psi \|_\infty \leq 1$. Then there exist constants $\gamma >0$ and $\beta >0$ such that
		\begin{align*} 
			\phi(x,v) =  \left( 1 -\gamma v\cdot \nabla_x M(x) -\beta \gamma \psi(v \cdot \nabla_x M(x)) v \cdot \nabla_x M(x)\right)e^{-\gamma M(x)}, 
		\end{align*} is a function for which the semigroup generated by $\mathcal{L}$ in \eqref{eq:L_generator} satisfies the Foster-Lyapunov condition \eqref{con:Foster-Lyapunov} with  $\beta = \chi/(1+\chi)$ and
		\begin{align*}
			\gamma \leq \min \left\{\frac{\tilde{\lambda}\chi(1-\chi)\xi}{8(1+\chi)} , \frac{1+\chi}{2(2+\chi)V_0 \|\nabla_x M\|_\infty} \right\},
		\end{align*} with \begin{align*} \xi:=
			\begin{cases}
				m_*^{k-2}, \quad &\mbox{if } k<2,\\ 
				1, \quad &\mbox{if } k=2, \\
				\|\nabla_x M\|_\infty^{k-2}, \quad &\mbox{if } k>2,
			\end{cases}
		\end{align*} where $m_*>0$ is found in \textbf{\em (H2)}.
	\end{lem}
	
	\begin{proof}
		We begin by a brief motivation of the form of $\phi$ in the proof. It is structurally similar to an estimate in Lemma 2.2 in \cite{MW17}. As the confining terms are bounded, we expect that we need to look for a Foster-Lyapunov functional which has exponential tails, by analogy with parabolic reaction diffusion equations with bounded drift terms. We can also guess this form by looking at the previous results on similar equations including \cite{MW17}. We choose a function of $M$ which will have this behaviour, $e^{-\gamma M}$, and seek a Foster-Lyapunov functional which is closely related to this. We derive the precise form of $\phi$ by repeatedly differentiating $\int f(t,z)e^{-\gamma M(x)}\d z$ along the flow of the equation until we find a term which doesn't change sign. We then create our Foster-Lyapunov function from a combination of $e^{-\gamma M(x)}$ and the key terms appearing in the derivatives of this moment along the flow of the equation.
		
		First we compute the action of $\mathcal{L}^*$ on the different elements.
		\begin{align*}
			\mathcal{L}^* \left( e^{-\gamma M(x)} \right) = -\gamma v \cdot \nabla_xM(x) e^{-\gamma M(x)}.
		\end{align*}
		Furthermore,
		\begin{align*}
			\mathcal{L}^* \left( v\cdot \nabla_x M(x) e^{-\gamma M(x)}\right) &= \left(v^T \hess (M)(x)v-\gamma (v\cdot \nabla_x M(x))^2  \right)e^{-\gamma M(x)}
			\\&- \left( (1-\chi \psi(v \cdot \nabla_x M(x)))v \cdot \nabla_x M(x)\right)e^{-\gamma M(x)}.
		\end{align*} Lastly,
		\begin{align*} 
			\mathcal{L}^* &\left(\psi(v \cdot \nabla_x M(x)) v \cdot \nabla_x M(x) e^{-\gamma M(x)} \right) = \left( \psi'(v \cdot \nabla_x M(x))v^T \hess (M(x)) v v \cdot \nabla_x M(x) \right) e^{-\gamma M(x)}\\
			&+\left(\psi(v \cdot \nabla_x M(x)) v^T \hess (M)(x) v - \gamma \psi(v \cdot \nabla_x M(x)) (v \cdot \nabla_x M(x))^2 \right)e^{-\gamma M(x)} \\
			&+ (1-\chi \psi (v \cdot \nabla_x M(x))) \left( \int_{\mathcal{V}} \psi(v' \cdot \nabla_x M(x)) v' \cdot \nabla_x M(x) \d v' - \psi(v \cdot \nabla_xM(x))v\cdot \nabla_x M(x)\right) e^{-\gamma M(x)}.
		\end{align*}
		Putting everything together gives,
		\begin{align} \label{ineqL*}
			\begin{split}
			\mathcal{L}^*\Big((1&-\gamma v \cdot \nabla_x M(x) -\beta \gamma \psi(v \cdot \nabla_x M(x)) v\cdot \nabla_x M(x))e^{-\gamma M(x)}\Big) 
			\\ \leq  &- \left( \beta\gamma(1-\chi )) \int_{\mathcal{V}} \psi(v'\cdot \nabla_x M(x)) v' \cdot \nabla_x M(x) \d v' \right ) e^{-\gamma M(x)} \\
			& +  \left( \beta \gamma(1+\chi) - \gamma \chi  \right )  \psi(v \cdot \nabla_x M(x)) v\cdot \nabla_x M(x)  e^{-\gamma M(x)}  \\
			& + \left(  \gamma^2(v\cdot \nabla_x M(x))^2 + \gamma^2 \beta \psi(v \cdot \nabla_x M(x))(v\cdot \nabla_x M(x))^2   \right ) e^{-\gamma M(x)}  \\
			& - \left ( \gamma +\beta \gamma \psi'(v \cdot \nabla_x M(x))v\cdot \nabla_x M(x) + \beta \gamma \psi(v \cdot \nabla_x M(x) ) \right)v^T\hess (M)(x)v e^{-\gamma M(x)}.
			\end{split}
		\end{align}
		We also have (for $\beta \leq 1$)
		\begin{multline*}
			-\gamma -\beta \gamma \psi'(v \cdot \nabla_x M(x))v\cdot \nabla_x M(x) - \beta \gamma \psi(v \cdot \nabla_x M(x)) 
			\\\leq \gamma + \beta \gamma \sup_{|z| \leq V_0 \|\nabla_x M\|_\infty} \left( \psi'(z)z + \psi(z)\right) \leq \gamma C_1(\psi, \|\nabla_x M\|_\infty),  
		\end{multline*}
		and
		\begin{align*}
			\gamma^2(v\cdot \nabla_x M(x))^2 + \gamma^2 \beta \psi(v \cdot \nabla_x M(x))(v\cdot \nabla_x M(x))^2 \leq 2\gamma^2 |\nabla_x M |^2.
		\end{align*}
		Combining these and choosing $\beta = \chi/(1+\chi)$, so that the second term on the right hand side of the inequality \eqref{ineqL*} vanishes, we have, 
		\begin{multline*}
			\mathcal{L}^*\Big((1-\gamma v \cdot \nabla_x M(x) -\beta \gamma \psi(v \cdot \nabla_x M(x))  v\cdot \nabla_x M(x))e^{-\gamma M(x)}\Big) 
			\\ \leq \left( -\beta \gamma\tilde{\lambda} (1-\chi) |\nabla_x M|^k 
			+ 2\gamma^2 |\nabla_x M|^2 + \gamma C_1 v^T \hess (M)(x) v \right)e^{-\gamma M(x)}.
		\end{multline*}
		Let us define
		\begin{align*} \xi:=
			\begin{cases}
				m_*^{k-2}, \quad &\mbox{if } k<2,\\ 
				1, \quad &\mbox{if } k=2, \\
				\|\nabla_x M\|_\infty^{k-2}, \quad &\mbox{if } k>2,
			\end{cases}
		\end{align*} where $m_*$ is coming from \textbf{(H2)}. Then, if we choose $\gamma$ so that the term with $\gamma^2$ will be controlled by the negative terms, and so that $\phi$ will be positive i.e. \[\gamma \leq \min \left\{\frac{\tilde{\lambda}\chi(1-\chi) \xi}{8(1+\chi)}, \frac{1 + \chi }{2 (2+\chi)V_0 \|\nabla_x M\|_\infty} \right\},\] then we have, at least for $x$ sufficiently large in the case $k<2$ that,
		\begin{multline*}
			\mathcal{L}^*\Big ((1-\gamma v \cdot \nabla_x M(x) -\beta \gamma \psi(v \cdot \nabla_x M(x))v\cdot \nabla_x M)e^{-\gamma M(x)}\Big) 
			\\ \leq \gamma \left(-\frac{3\tilde{\lambda}\chi(1-\chi) }{8(1+\chi)}|\nabla_x M|^k +C_1 V_0^2|\hess (M) (x)| \right)e^{-\gamma M(x)}.
		\end{multline*}
		Then by hypothesis \textbf{(H2)} there exist $R>0$ and $m_*>0$ such that when $|x|>R$ we have
		\begin{align} \label{R}
			|\nabla_x M| > m_*, \quad \mbox{and}  \quad |\hess (M) (x) | \leq \frac{\tilde{\lambda} \chi (1-\chi) m_*^k}{4C_1 (1+\chi) V_0^2}. 
		\end{align} So we have,
		\begin{align*}
			\mathcal{L}^*\left((1-\gamma v \cdot \nabla_x M(x) -\beta \gamma \psi(v \cdot \nabla_x M(x)) v  \cdot \nabla_x M(x))e^{-\gamma M(x)}\right) \\
			\leq A\mathds{1}_{\{|x|<R\}} - \frac{\gamma\tilde{\lambda}\chi(1-\chi)m_*^k}{8 (1+\chi)} e^{-\gamma M(x)},
		\end{align*}
		where 
		\begin{align}
			\label{A}
			A = \sup_{|x| \leq R} \left \{ \gamma C_1 V_0^2 |\mbox{Hess}(M)(x)|e^{-\gamma M(x)} \right \}. 
		\end{align}
		Since we can compare $e^{-\gamma M(x)}$ to $(1-\gamma v \cdot \nabla_x M(x)- \beta \gamma \psi(v \cdot \nabla_x M(x))v \cdot \nabla_x M(x))e^{-\gamma M(x)}$ by
		\begin{align*} 
			(1-\gamma v \cdot \nabla_x M(x)- \beta \gamma \psi(v \cdot \nabla_x M(x))v \cdot \nabla_x M(x))e^{-\gamma M(x)} & \leq \left(1+ \gamma V_0 \|\nabla_x M\|_\infty (1+\beta \|\psi\|_\infty) \right) e^{-\gamma M(x)}\\ & \leq \frac{3}{2}e^{-\gamma M(x)}, 
		\end{align*} if we write 
		\[ \phi(x,v) = (1-\gamma v \cdot \nabla_x M(x)- \beta \gamma \psi(v \cdot \nabla_x M(x))v \cdot \nabla_x M(x))e^{-\gamma M(x)}, \] then
		\begin{align}
			\label{ineq:L_*_phi}
			\mathcal{L}^*\phi \leq A -  \frac{\gamma \tilde{\lambda}\chi (1-\chi)m_*^k}{12 (1+\chi)} \phi =  -  \frac{\gamma \tilde{\lambda}\chi (1-\chi)m_*^k}{12 (1+\chi)} \left(A' - \phi \right), 
		\end{align} where
		\begin{align*}
			A' =\frac{12 C_1 V_0^2 (1+\chi)}{\tilde{\lambda}\chi(1-\chi)m_*^k}\sup_{|x| \leq R}\left \{ |\mbox{Hess}(M)(x)| e^{-\gamma M(x)} \right \}.
		\end{align*}
		Therefore
		\begin{align*}
			\int f (t,z)\phi (z) \d z\leq A' + \exp \left( -  \frac{\gamma \tilde{\lambda}\chi (1-\chi)m_*^k}{12 (1+\chi)} t \right) \left(\int f_0 (z)\phi(z) \d z - A' \right). 
		\end{align*}
		Thus we prove \eqref{ineq:lyapunov1} for $\alpha = \exp \left( -  \frac{\gamma \tilde{\lambda}\chi (1-\chi)m_*^k}{12 (1+\chi)} t \right) $ and $C = A'$. 
	\end{proof}

	\subsection{Minorisation condition} \label{sec:minorisation}
	In this section, we show the minorisation condition \eqref{con:minorisation} for Equation \eqref{eq:rt-linear}. We consider two semigroups $(\mathcal{T}_t)_{t \geq 0}$ and $(\mathcal{S}_t)_{t \geq 0}$. Let $(\mathcal{T}_t)_{t \geq 0}$, represents the \emph{transport} part, be associated to the equation
	\begin{align}
		\label{eq:D1}
		\p_t f + v \cdot \nabla_x f + \lambda (x,v) f=0,
	\end{align} which means that the solution of \eqref{eq:D1} can be written as for $t \geq 0$, for all $x\in \R^d$,
	\begin{align} \label{bdP}
	%	\begin{split}
			\mathcal{T}_t f_0(x,v)= 
		%	\begin{cases}
				f_0(x-vt, v) %\quad &x \geq vt\\
			%	0, \quad &x < vt.
	%		\end{cases}
	%	\end{split}
	\end{align}
	Here since $f_t$ is a measure we understand the change of variables $x \mapsto x-vt$ by duality. 
	Let $(\mathcal{S}_t)_{t \geq 0}$ be associated to the equation
	\begin{align} \label{eq:D2}
		\p_t f + v \cdot \nabla_x f + \lambda (v, x) f = \int_{\mathcal{V}} \lambda ( x,v')  f(t,x, v') \d v'.
	\end{align}
	Then the solution of \eqref{eq:D2} is
	\begin{align*}
		f(t,x,v) = \mathcal{S}_t f_0(x,v) = e^{- \int_0^t \lambda(x-vs) \mathrm{d}s}\mathcal{T}_t f_0(x,v) + \int_{0}^{t}  e^{-\int_s^t \lambda(x-vr) \mathrm{d}r }\mathcal{T}_{t-s} (\mathcal{J}f (s, x,v) ) \d s ,
	\end{align*} where $\mathcal{J} f(t,x,v) := \int_{\mathcal{V}} \lambda (x,v') f(t,x, v') \d v'$ is the \emph{jump operator}. Remark that we have 
	\begin{align} \label{J_bd}
		\mathcal{J} f(t,x,v) = \int_{\mathcal{V}} \lambda (x,v') f(t,x, v') \d v' \geq (1-\chi)  \mathds{1}_{\{|v| \leq V_0\}}\int_{\mathcal{V}}  f(t,x, v') \d v'. 
	\end{align}
Since $\mathcal{S}_t f_0(x,v) \geq  e^{-(1-\chi)t}\mathcal{T}_t f_0(x,v)$ and $\mathcal{S}_t f_0(x,v)  \geq \int_{0}^{t}e^{-(1+\chi)(t-s)} \mathcal{T}_{t-s} (\mathcal{J}f (s, x,v) ) \d s$, we substitute the first of these inequalities into the second and then iteratively substitute the result into the second to get 
	\begin{align*}
		f(t,x,v)   = \mathcal{S}_t f_0(x,v) \geq   (1-\chi)^2 e^{-(1+\chi)t }  \int_{0}^{t} \int_{0}^{s}  \mathcal{T}_{t-s} \mathcal{J} \mathcal{T}_{s-r} \mathcal{J} \mathcal{T}_r f_0 (x,v) \d r \d s.
	\end{align*}
	\begin{lem}\label{lem:transport}
		Given any time $t_0>0$, for all $t\geq t_0$ it holds that 
		\begin{align*}
			\int_{\mathcal{V} } \mathcal{T}_t \left( \delta_{x_0}(x) \mathds{1}_{\{|v_0| \leq V_0\}}(v) \right) \d v \geq   e^{-(1+\chi)t} \frac{1}{t^d|B(V_0)|} \mathds{1}_{\{|x-x_0| \leq V_0t\}}  \quad \text{for any } x_0 , v_0>0.
		\end{align*}
	\end{lem}
	\begin{proof} 
		Note that we have 
		\begin{align*}
			\mathcal{T}_t f_0(x,v) \geq e^{-(1+\chi)t} f_0(x-vt, v), \quad t \geq 0.
		\end{align*} 
		For an arbitrary starting point and a velocity $(x_0,v_0)$, $x_0 >0$, $v_0 \in B(V_0)$ (ball of radius $V_0$) we have 
		\begin{align*}
			\mathcal{T}_t \left( \delta_{x_0}(x) \mathds{1}_{\{|v_0| \leq V_0\}}(v)\right) \geq  e^{-(1+\chi)t}  \delta_{x_0}(x-vt) \mathds{1}_{\{|v_0| \leq V_0\}}.
		\end{align*} By integrating this and changing variables we obtain 
		\begin{align*}
			\int_{\mathcal{V} } \mathcal{T}_t \left( \delta_{x_0}(x) \mathds{1}_{\{|v_0| \leq V_0\}}\right) \d v &\geq e^{-(1+\chi)t}  \int_{\mathcal{V} } \delta_{x_0}(x-vt) \mathds{1}_{\{|v_0| \leq V_0\}}(v)\d v \\
			&\geq e^{-(1+\chi)t} \frac{1}{t^d|B(V_0)|} \int_{\left |\frac{x-y}{t} \right | \leq V_0 } \delta_{x_0}(y) \mathds{1}_{\left \{ \left |\frac{x-y}{t} \right | \leq V_0 \right \}}(v)\d y.
		\end{align*} This gives the result.
	\end{proof}
	Now, we prove the minorisation condition for \eqref{eq:rt-linear} below.
	\begin{lem}[Minorisation condition for Equation \eqref{eq:rt-linear}]  \label{lem:Minorisation}
		For every $R_*>0$ we can take $t = 3 + R_*/V_0$ such that any solution of Equation \eqref{eq:rt-linear} with initial data $f_0 \in \mathcal{P}(\R^d \times \mathcal{V})$ with $\int_{|x|\leq R_*} \int_{\mathcal{V}}f_0(x,v) \d x \d v =1$ satisfies 
		\begin{align}
			f(t, x,v) \geq (1-\chi^2) e^{-(1+\chi)t} \frac{1}{t^d |B(V_0)|} \mathds{1}_{\{|x| \leq V_0\}}\mathds{1}_{\{|v| \leq V_0\}}.
		\end{align}
	\end{lem}
	\begin{proof}
		We take $f_0 (x,v): = \delta _{(x_0,v_0)} $ where $(x_0,v_0) \in\R^d \times \mathcal{V}$, is an arbitrary point with an arbitrary velocity. We only need to consider $x_0 \in B(0, R_*)$, then the bound we obtain depends on $R_*$.
		First we have that 
		\begin{align*}
			\mathcal{T}_rf_0 \geq e^{-(1+\chi)r} \delta_{(x_0 + rv_0, v_0)}.
		\end{align*}
		Applying $\mathcal{J}$ to this we get
		\begin{align*}
			\mathcal{J} \mathcal{T}_rf_0 \geq (1-\chi) e^{-(1+\chi)r} \delta_{x_0+rv_0}(x) \mathds{1}_{\{|v| \leq V_0\}}.
		\end{align*}
		We then apply Lemma \ref{lem:transport} and obtain
		\begin{align*}
			\int_{\mathcal{V}} \mathcal T_{s-r} \mathcal J  \mathcal T_r f_0 \geq  (1-\chi) e^{-(1+\chi)s} \frac{1}{(s-r)^d |B(V_0)|} \mathds{1}_{\{|x-x_0-rv_0| \leq V_0 (s-r)\}}.
		\end{align*} This means that 
		\begin{align*}
			\mathcal{J} \mathcal{T}_{s-r}\mathcal{J}\mathcal{T}_r f_0 \geq (1-\chi)^2 e^{-(1+\chi)s}   \frac{1}{(s-r)^d |B(V_0)|}  \mathds{1}_{\{|x-x_0-rv_0| \leq V_0 (s-r)\}}   \mathds{1}_{\{|v| \leq V_0\}}.
		\end{align*} 
		Lastly we have that 
		\begin{align*}
			\mathcal{T}_{t-s} \mathcal{J}\mathcal{T}_{s-r}\mathcal{J}\mathcal{T}_r f_0 \geq (1-\chi)^2 e^{-(1+\chi)t} \frac{1}{(s-r)^d |B(V_0)|}  \mathds{1}_{\{|x-(t-s)v-x_0-rv_0| \leq V_0 (s-r)\}}   \mathds{1}_{\{|v| \leq V_0\}}.
		\end{align*}
		Since we have (remembering that all the velocities are smaller than $V_0$)
		\begin{align*}
			|x-v(t-s)-x_0 -rv_0| \leq (s-r) V_0,
		\end{align*}implies that
		\begin{align*}
			|x| \leq  (s-r) V_0 - (t-s)V_0 - r V_0 - R_*. 
		\end{align*}
		Then if we ensure that $(s-r) \geq 2 + R_*/V_0$, $ r \leq 1/2$ and $(t-s) \leq 1/2$ we will have
		\begin{align*}
			\mathcal{T}_{t-s} \mathcal{J}\mathcal{T}_{s-r}\mathcal{J}\mathcal{T}_r f_0 \geq (1-\chi)^2 e^{-(1+\chi)t} \frac{1}{(s-r)^d |B(V_0)|}  \mathds{1}_{\{|x| \leq V_0\}}  \mathds{1}_{\{|v| \leq V_0\}}.
		\end{align*}
		Therefore let us set $t = 3+R_*/V_0$. Then we can restrict the time integrals to $r \in (0,1/2)$, $s \in (5/2 + R_*/V_0, 3 + R_*/V_0)$. Then we get 
		\begin{align*}
			f(t,x,v)&\geq \int_{0}^{t} \int_{0}^{s} \mathcal{T}_{t-s} \mathcal{J}\mathcal{T}_{s-r}\mathcal{J}\mathcal{T}_r f_0 (x,v)\d r \d s \\
			&\geq (1-\chi)^2 e^{-(1+\chi)t}  \int_{5/2 + R_*/V_0}^{3 + R_*/V_0} \int_{0}^{1/2}\frac{1}{(s-r)^d |B(V_0)|} \mathds{1}_{\{|x| \leq V_0\}}  \mathds{1}_{\{|v| \leq V_0\}} \d r \d s\\
			&\geq  (1-\chi)^2 e^{-(1+\chi)t}   \frac{1}{t^d |B(V_0)|} \mathds{1}_{\{|x| \leq V_0\}}  \mathds{1}_{\{|v| \leq V_0\}}.
		\end{align*}
		This gives the uniform lower bound we need for Harris's theorem. We can extend this from delta function initial data to general initial data by using the fact that the associated semigroup is Markov. 
	\end{proof}
	
	\begin{proof}[Proof of Theorem \ref{thm1}] We verify the two hypotheses of Harris's theorem in Lemmas \ref{lem:Foster-Lyapunov} and \ref{lem:Minorisation}. The contraction in the $\|\cdot\|_*$ norm and the existence of a steady state follow again by  Harris's theorem.  
		
		Moreover Lemma \ref{lem:Foster-Lyapunov} gives that for the steady state $f_\infty$ obtained by Harris's theorem we have
		\begin{align*}
			\int \phi(z) f_\infty(z) \d z \leq A' . 
		\end{align*} Our conditions on $\gamma$ ensure that 
		\begin{align*}
			\frac{1}{2} e^{-\gamma M(x)} \leq  \phi \leq \frac{3}{2} \phi.
		\end{align*}
		Therefore we obtain
		\begin{align*}
			\int  e^{-\gamma M(x)} f_\infty(z) \d z \leq 2A',
		\end{align*}
		and this leads to 
		\begin{align*}
			\int e^{-\gamma M(x)}  f(t,z) \d z \leq 2A' + 3  \exp \left( -  \frac{\gamma \tilde{\lambda}\chi (1-\chi)m_*^k}{6 (1+\chi)} t \right) \int e^{-\gamma M(x)}  f_0 (z) \d z, 
		\end{align*} which gives the contraction in the $\|\cdot\|_{**}$ norm.
		We remark that in this proof $\gamma$ only depends on $M$ through $\tilde{\lambda}$ and $\|\nabla_x M\|_\infty$. So if $\psi'(0)>0$ we can choose $\gamma$ uniformly over sets of $M$ where $\nabla_x M$ is bounded uniformly.
		
	\end{proof}

	\section{Weakly non-linear coupling} \label{sec:non-linear}
	
	\subsection{Stationary solutions}
	In this section, we build a stationary state for the run and tumble equation \eqref{eq:rt-linear} with the weakly non-linear coupling \eqref{eq:weakly_nonlin}. We know by Theorem \ref{thm1} that there exists a unique steady state solution to the linear equation satisfying the assumptions listed in Theorem \ref{thm1}. For each fixed  $M$, we call $\mathcal{S}_t^M$ the semigroup on measures associated to the linear equation and $f_{\infty}^M$ its unique stationary solution. Then we see that $f_{\infty}^M$ satisfies 
	\begin{align} \label{eq:rt_stat} 
		v \cdot \nabla_x f_{\infty}^M(x,v)+ \lambda (v \cdot \nabla_x M(x)) f_{\infty}^M(x,v)- \int \lambda (v' \cdot \nabla_x M(x)) f_{\infty}^M(x, v') \d v' =0.
	\end{align}
	We define a function $G : \mathcal B \rightarrow C^2(\mathbb{R})$, where $\mathcal{B}$ is the set of $M$ satisfying \textbf{(H2)} given by
	\begin{align} \label{def:G}
		G(M) = \log\left(S_\infty \left(1 + \eta N * \rho^M \right)\right),
	\end{align} where $S_{\infty}$ a smooth function, having exponential tails with some fixed parameter, $\eta>0$ a small constant, $N$ a positive, compactly supported, smooth function, and $\rho^M := \int f_{\infty}^M(x,v)\d v$.
	We see that if $M$ is a fixed point of $G$ then $f_{\infty}^M$ will be a steady state of the non-linear equation.
	\begin{prp} \label{prop:unifeta}
		Suppose that $M$ is of the form $M = M_\infty + \log \left( 1+ \eta N*\rho\right)$ for some $\rho \in \mathcal{P}(\R^d)$. Then if $\eta$ is small enough in terms of $\|N\|_{W^{2,\infty}}$, we have that 
		\begin{align*}
			\|\mathcal{S}_t^M  f\|_{**} \leq De^{-\sigma t} \| f \|_{**},
		\end{align*} where $D, \sigma$ are strictly positive constants only depending on $M_\infty, N,$ and  $\eta$. Furthermore, if $f_{\infty}^M$ is the steady state of $\mathcal{S}_t^M$ then 
		\begin{align} \label{bd:f_infty}
			\|f_{\infty}^M\|_{**} \leq \tilde{C},
		\end{align} where $\tilde{C}$ is a constant depening on $M_\infty, N$, $\eta$ and $\|\cdot\|_{**}$ is defined in \eqref{eq:norm2}. 
	\end{prp}
	\begin{proof}
		The result follows from Theorem \ref{thm1}. We recall that the constants in Lemma \ref{lem:Minorisation} in the minorisation part do not depend on $M$, whereas, the constants in Lemma \ref{lem:Foster-Lyapunov} in the Foster-Lyapunov part depend on $M$ through $ \| \nabla_x M\|_\infty, R,$ and $m_*$ so that for all $|x|>R$ we have (recalling \eqref{R}),
		\begin{align*}
			|\nabla_x M|> m_*, \quad \mbox{and}\quad |\hess(M)| \leq \frac{\tilde{\lambda}\chi (1-\chi) m_*^k}{4C_1 (1+\chi)V_0^2}. 
		\end{align*}We want to verify this for $M$ solving \eqref{eq:weakly_nonlin}. We can control $|\nabla_x M|$ and $|\hess(M)|$ by considering
		\begin{align*}
			M = M_\infty + \log \left(1 + \eta N* \rho\right) \sim M_\infty + \eta N*\rho^M. 
		\end{align*}
		Provided that $\eta \leq \|N\|^{-1}_\infty$, which we can choose it to be, by Taylor expansion we have that
		\begin{align*}
			|M-M_\infty| \leq \eta N*\rho \leq \eta \|N\|_\infty. 
		\end{align*}
		In a similar way, we can take gradients to get
		\begin{align*}
			\nabla_x M = \nabla_x M_\infty + \eta \frac{\nabla_x N * \rho}{1+ \eta N * \rho}. 
		\end{align*} Then
		\begin{align*}
			\left \| \frac{\nabla_x N * \rho}{1+ \eta N*\rho} \right \|_\infty \leq \|\nabla_x N * \rho\|_{\infty} \leq  \|\nabla_x N\|_\infty.
		\end{align*} So we can ensure that
		\begin{align}
			\label{bds:grad_M-grad_M_inf}
			|\nabla_x M - \nabla_x M_\infty | \leq \eta\|\nabla_x N\|_\infty. 
		\end{align}
		We can also compute the Hessian to get
		\begin{align*}
			\hess(M) = \hess(M_\infty)  + 
			\frac{ \eta (\hess (N)*\rho )+ \eta^2 \left( (N*\rho)(\hess (N)*\rho) - (\nabla_x N * \rho)(\nabla_x N^T * \rho) \right)}{(1+ \eta N*\rho)^2}.
		\end{align*}
		Therefore, the difference between $\hess(M)$ and $\hess(M_\infty)$ is controlled by $\eta \|N\|_{W^{2,\infty}}$.
		Suppose that there exist $R_\infty$ and $m_\infty$ such that for all $|x|>R_\infty$ we have
		\begin{align*}
			|\nabla_x M_\infty| \geq m_\infty, \quad \mbox{and} \quad |\hess(M_\infty)| \leq  \frac{\tilde{\lambda}\chi (1-\chi) m_\infty^k}{32C_1 (1+\chi)V_0^2}.
		\end{align*} Then by choosing $\eta$ small enough in terms of $\|N\|_{W^{2, \infty}}$ and setting $m_* = m_\infty/2$ and $R = R_\infty$ we have $m_*$ and $R$ in \eqref{R} only depend on $M_\infty, N, \eta$.
		
		Furthermore, by Theorem \ref{thm1} for the steady state $f_\infty^M$ we have
		\[ \int e^{-\gamma M(x)} f_\infty^M (z)\d z \leq 2A' , \] where
		\begin{align*}
			A' =\frac{6 C_1 V_0^2 (1+\chi)}{\tilde{\lambda}\chi(1-\chi)m_*^k}\sup_{|x| \leq R}\left \{ |\mbox{Hess}(M)(x)| e^{-\gamma M(x)} \right \}.
		\end{align*}
		We can bound $A'$ only in terms of $M_\infty, N, \eta$. We already know this is true for  $m_*$ and $R$.  Moreover, as $\gamma \leq 1$ we have
		\begin{align*}
			\sup_{|x| \leq R} \left \{|\hess(M)(x)| e^{-\gamma M(x)} \right \} \leq \sup_{|x| \leq R} \left \{ \left( |\hess(M_{\infty})(x)| + \eta \|N\|_{W^{2, \infty}} \right) e^{- M_\infty(x) + \eta \|N \|_\infty} \right \},
		\end{align*} which we can bound in a way that only depends on $M_\infty, N, \eta$. Therefore,
		\begin{align*}
			\int e^{-\gamma M_\infty(x)}  f_\infty^M (z)\d z \leq e^{\eta \|N\|_\infty} \int e^{-\gamma M(x)} f_\infty^M (z)\d z, 
		\end{align*} and we can compare $\gamma M_\infty(x)$ to $\delta$ in Theorem \ref{thm1}. So this lets us control $\|f_\infty^M\|_{**}$ in terms of $A'$ up to factors only depending on $M_\infty, N, \eta$. This finishes the proof.
	\end{proof}
	Then we can prove
	\begin{prp}  \label{prop:ss_nonlin}
		We consider Equation \eqref{eq:rt-linear} with the weakly non-linear coupling \eqref{eq:weakly_nonlin} where we suppose that $N$ is a positive, smooth function with a compact support, $\eta>0$ is a constant, and $S_\infty$ is a smooth function satisfying for some $\ushort C , \bar C, \alpha >0$ that
		\begin{align} \label{ineq:M_bounds}
			\ushort C -\alpha \langle x \rangle \leq M_\infty (x): =\log(S_\infty(x))\leq \bar C - \alpha \langle x \rangle, 
		\end{align} where  $\langle x\rangle = \sqrt{1 + x^2}$. Then there exists some constant $\tilde{C}$ depending on $\ushort C, \bar C, \alpha$ such that if $\eta < \tilde{C}$ then $G$ has a unique fixed point, $\tilde{M}$ and there exists a unique steady state solution to Equation \eqref{eq:rt-linear} with a weakly non-linear coupling, $F_\infty$.
	\end{prp}
	\begin{proof} We want to use the contraction mapping theorem to show that $G$, defined by \eqref{def:G}, has a fixed point. Let us take for $i = \{1,2\}$, 
		\begin{align*}
			M_i = M_\infty + \log\left(1+ \eta N*\rho^{M_i}\right), \quad \mbox{where} \quad \rho^{M_i} = \int_{\mathcal{V}} f_{\infty}^{M_i} (x,v)\dv.
		\end{align*} We also know that $M_\infty$ satisfies \eqref{ineq:M_bounds}. Then we  show contractivity of $G$ by using the fact that
		\begin{align*}
			\|G(M_1) - G(M_2)\|_\infty \leq C\eta \|  N * \rho^{M_1} - N*\rho^{M_2}\|_\infty \leq C \eta \|N\|_\infty \|f_\infty^{M_1}- f_\infty^{M_2}\|_{**},
		\end{align*} where  $C>0$ is a constant. 
		
		Let us call $\mathcal{S}_t^{M_i}$, for $i = \{1,2\}$, the semigroups associated to the linear equation with $M_i := \log S_i$. Then, we choose $t$ sufficiently large so that $\mathcal{S}^{M_1}_t$ is a contraction. By Proposition \ref{prop:unifeta} we know that there exist $D$,$\sigma>0$ such that 
		\begin{align*}
			\|\mathcal{S}_t^{M_1}(f-g)\|_{**} \leq   D e^{-\sigma t} \| f - g \|_{**}.
		\end{align*}
		The constants $D, \sigma$ only depend on $M_\infty, N, \eta$ because it was shown in Lemma \ref{prop:unifeta}, the bounds on $M$ required to prove Theorem \ref{thm1} are preserved by $G$ and do not depend on $M$ except through, $M_\infty, N, \eta$. We recall
		\begin{align*}
			\|f\|_{**} = \int_{\R^d} \int_{\mathcal{V} } e^{\delta \langle x \rangle}  |f(t, x,v)| \d x \d v, 
		\end{align*} where $\delta = \beta \gamma$. Note that the defnition of $\delta$ comes from the fact that we essentially weight by $e^{-\gamma M_\infty(x)}$ and $M_\infty(x) \sim -\beta \langle x \rangle$. 
		Let us call $f_{\infty}^{M_i}$ the steady state solutions of the linear equation with $M_i$ for $i=\{1,2\}$. Then
		\begin{align*}
			\|f_{\infty}^{M_1} - f_{\infty}^{M_2} \|_{**} = \|\mathcal{S}_t^{M_1}f_{\infty}^{M_1} - \mathcal{S}_t^{M_2}f_{\infty}^{M_2}\|_{**} \leq \|\mathcal{S}_t^{M_1}(f_{\infty}^{M_1}-f_{\infty}^{M_2})\|_{**} + \|(\mathcal{S}_t^{M_1} - \mathcal{S}_t^{M_2}) f_{\infty}^{M_2}\|_{**} 
		\end{align*} leading to
		\begin{align} \label{ineq:f_1-f_2}
			(1-De^{-\sigma t}) \|f_{\infty}^{M_1}-f_{\infty}^{M_2}\|_{**} \leq \| (\mathcal{S}_t^{M_1} - \mathcal{S}_t^{M_2}) f_{\infty}^{M_2}\|_{**}.
		\end{align}
		So it only remains to show  that for a fixed time period, $\mathcal{S}_t^M$ is continuous in $M$. 
		
		Let us write
		\[ \Lambda (s, t, M_i)(x,v) = \int_s^t \lambda(v \cdot \nabla_x M_i(x-v(t-r)))\d r, \] and
		\[ \mathcal{J}^{M_i}(f)(x,v) = \int_{\mathcal{V}}\lambda(v'\cdot \nabla_xM_i (x)) f(x,v')\d v'. \] Then we have
		\[ \mathcal{S}_t^{M_i} f = e^{-\Lambda (0,t,M_i)}\mathcal{T}_t f + \int_0^t e^{-\Lambda(s,t,M_i)} \mathcal{J}^{M_i} \mathcal{T}_{t-s}\mathcal{S}^{M_i}_s f \d s,\] where $(\mathcal{T})_{t \geq 0}$ is defined in \eqref {eq:D2}. Consequently we have 
		\begin{align*}
			|\mathcal{S}_t^{M_1}f - \mathcal{S}_t^{M_2}f| &\leq \left( e^{-\Lambda(0,t, M_1)}-e^{-\Lambda(0,t, M_2)}\right)\mathcal{T}_t f  \\
			& \quad + \int_0^t \left( e^{-\Lambda (s,t M_1)} - e^{-\Lambda (s, t, M_2)} \right) \mathcal{J}^{M_1}\mathcal{T}_{t-s}\mathcal{S}_s^{M_1}f \d s \\
			& \quad + \int_0^t e^{-\Lambda(s,t, M_2)}(\mathcal{J}^{M_1} - \mathcal{J}^{M_2}) \mathcal{T}_{t-s}\mathcal{S}^{M_1}_s f \d s \\
			& \quad + \int_0^t e^{-\Lambda(s,t,M_2)}\mathcal{J}^{M_2} \left( \mathcal{S}_s^{M_1} - \mathcal{S}^{M_2}_s\right)f \d s.
		\end{align*} 
		We can see that for $s,t \leq T$ there exists a constant $C_T >0$ depending on $T$ so that
		\begin{align*}
			\left| e^{-\Lambda (s,t M_1)} - e^{-\Lambda (s, t, M_2)} \right| \leq C_T \|\nabla_x M_1 - \nabla_x M_2\|_\infty.  
		\end{align*} We also have trivially that
		\[ e^{-\Lambda (s,t, M)} \leq 1. \] 
		Turning to the jump operator $\mathcal{J}$ we have
		\begin{align*}
			\|(\mathcal{J}^{M_1} - \mathcal{J}^{M_2})f\|_{**} \leq \| \lambda (v \cdot \nabla_x M_1) - \lambda (v \cdot \nabla_x M_2) \|_\infty \|f\|_{**} \leq C\|\nabla_x M_1 - \nabla_x M_2\|_\infty \|f \|_{**},
		\end{align*} and
		\begin{align*}
			\|\mathcal{J}^{M_i}f\|_{**} \leq (1+ \chi) \|f\|_{**}. 
		\end{align*} We also have
		\[ \|\mathcal{T}_t f \|_{**} \leq e^{2 \delta V_0 t} \|f\|_{**}. \] Therefore we obtain, for $t \leq T$,
		\begin{align*} 
			\left   \| \left( \mathcal{S}^{M_1}_t - \mathcal{S}^{M_2}_t  \right) f  \right  \|_{**} \leq C_T \|\nabla_x M_1 - \nabla_x M_2 \|_\infty \|f \|_{**}+ \int_0^t C_T  \left \| \left( \mathcal{S}^{M_1}_s - \mathcal{S}^{M_2}_s \right) f \right \|_{**} \d s   
		\end{align*}
		Then Gronwall's inequality gives,
		\begin{align}
			\label{ineq:S_1-S_2}
			\left \| \left( \mathcal{S}^{M_1}_t - \mathcal{S}^{M_2}_t\right) f \right \|_{**} \leq C'_T \|\nabla_x M_1 - \nabla_x M_2 \|_\infty \|f\|_{**},
		\end{align} where $C_T' >0$ a constant depending on $T$.
		
		Using \eqref{ineq:f_1-f_2} and \eqref{ineq:S_1-S_2} we obtain an estimate on the steady states given by
		\begin{align*}
			\| f_{\infty}^{M_1} - f_{\infty}^{M_2} \|_{**} \leq (1- De^{-\sigma T})^{-1} C'_T \|\nabla_x M_1 - \nabla_x M_2 \|_\infty \|f_{\infty}^{M_2} \|_{**}. 
		\end{align*}
		Now we can see that
		\begin{align} \label{est:rho_f}
			\|\rho^{M_1}- \rho^{M_2}\|_{**} = \|f_{\infty}^{M_1}  - f_{\infty}^{M_2} \|_{**}. 
		\end{align} Consequently we have,
		\[ \|G(M_1) - G(M_2)\|_\infty \leq  C \eta  \|\nabla_x M_1 - \nabla_x M_2 \|_\infty \| \rho^{M_2}\|_{**}. \] 
		Similarly
		\[ \|\nabla_x G(M_1) - \nabla_x G(M_2)\|_\infty \leq  C \eta \|\nabla_x M_1 - \nabla_x M_2 \|_\infty \|\rho^{M_2}\|_{**}. \] By Proposition \ref{prop:unifeta}, we also have that
		\begin{align*}
			\|\rho^{M_2}\|_{**} = \| f_{\infty}^{M_2}\|_{**} \leq \tilde{C}.
		\end{align*} So we choose $\eta$ sufficiently small  to get
		\begin{align*}
			\|G(M_1) - G(M_2)\|_{W^{1, \infty}} \leq \frac{1}{2} \|M_1 - M_2 \|_{W^{1, \infty}}.
		\end{align*} This gives a unique fixed point of $G$ which we call $\tilde M$ such that $G(\tilde M) = \tilde M$. Thus, $F_\infty = f^{\tilde{M}}$ is the unique steady state solution of the the weakly non-linear equation.
	\end{proof}

	\subsection{Perturbation argument} \label{sec:pertubation}
	
	In this section, we prove that the solution of Equation \eqref{eq:rt-linear} with the weakly non-linear coupling \eqref{eq:weakly_nonlin} converges exponentially to its unique steady state solution obtained in Proposition \eqref{prop:ss_nonlin}. We showed, in Proposition \ref{prop:unifeta}, that we can find $R$, $m_*$ and bound $\|\nabla_x M\|_\infty$ uniformly over the set of log-chemoattractants of the form
	\begin{align*}
		M = M_\infty + \log(1+ \eta N * \rho),
	\end{align*} for some probability density $\rho$ on $\mathbb{R}^d$. This means that we can also fix, $\gamma$ and $\tilde{\lambda}$ uniformly over this set since we show in the proof of Proposition \ref{prop:unifeta} that they only depend on these bounds.
	
	Let us first look at a moment estimate for the weakly non-linear equation \eqref{eq:rt-linear}-\eqref{eq:weakly_nonlin}. We would like to show an inequality analogous to  \eqref{ineq:lyapunov1} for the solution $F$ of the weakly non-linear equation. That is to say we show
	\begin{align} \label{ineq:lyapunov_nonlin}
		\int  e^{-\gamma M_\infty(x)}  F(t,z) \d z \leq \alpha \int e^{-\gamma M_\infty(x)}  F_0(z)  \d z  + C\int F_0(z) \d z.
	\end{align} 
	Let us define two operators $\mathcal{L}_{M_t}$ and $\mathcal{L}_{M_{\infty}}$ associated to the weakly non-linear equation and the equation for the stationary solution \eqref{eq:rt_stat} respectively. Then we have 
	\begin{align} \label{defn:L_M_t}
		\mathcal{L}_{M_t}f = - v\cdot \nabla_x f + \int \lambda( v'\cdot \nabla_x M) f(t,x,v')\d v' - \lambda( v \cdot \nabla_x M) f(t,x,v), 
	\end{align} where $M$ is given by \eqref{eq:weakly_nonlin}. 
	Similarly $\mathcal{L}_{M_{\infty}}$ is given by 
	\begin{align} \label{defn:L_M_inf}
		\mathcal{L}_{M_\infty}f= - v\cdot \nabla_x f + \int \lambda( v'\cdot \nabla_x M_\infty) f(t, x,v')\d v' - \lambda(v \cdot \nabla_x M_\infty) f(t,x,v). 
	\end{align}
	We carry out a similar argument to the one in Section \ref{sec:Foster-Lyapunov} for the linear equation. We show
	\begin{lem}\label{lem:perturb1} 
		Suppose that $\mathcal{L}_{M_t}$ and $\mathcal{L}_{M_\infty}$ are given by \eqref{defn:L_M_t} and \eqref{defn:L_M_inf} and $\mathcal{L}_{M_t}^*$, $\mathcal{L}_{M_\infty}^*$ denote their formal adjoints respectively. Then let
		\begin{align} \label{phi}
			\phi (x,v) = (1-\gamma v\cdot \nabla_x M(x)- \beta \gamma \psi(v \cdot \nabla_x M (x)) v \cdot \nabla_x M(x))e^{-\gamma M_\infty(x)}, 
		\end{align} and $M_t = M_\infty + \log(1+ \eta N*\rho_t)$ where $\rho_t := \int_{\mathcal{V}} F(t,x,v) \d v$. Then we have 
		\begin{align} \label{ineq:L_*-L_M}
			\mathcal{L}_{M_t}^*\phi \leq \mathcal{L}_{M_\infty}^*\phi + 4\eta \chi V_0\|\psi'\|_\infty \|\nabla_x N\|_\infty  e^{-\gamma M_\infty(x)}.
		\end{align}
	\end{lem}
	\begin{proof}
		First, using \eqref{bds:grad_M-grad_M_inf} we obtain
		\begin{align} \label{bound:Psi_M-Psi_M_inf}
			|\psi (v \cdot \nabla_x M_\infty) - \psi(v\cdot \nabla_x M)| \leq  \|\psi'\|_\infty |v | |\nabla_x M - \nabla_x M_{\infty}| \leq \eta V_0 \|\psi'\|_\infty   \|\nabla_x N\|_\infty. 
		\end{align} 
		Then, we see that 
		\begin{align*}
			\mathcal{L}_{M_t}^*  \phi - \mathcal{L}_{M_{\infty}}^*\phi &= (\lambda (v \cdot \nabla_x M_t) - \lambda (v \cdot \nabla_x M_\infty) )  \left(\int_{\mathcal{V}} \phi(x,v') \d v' - \phi (x,v)\right) \\
			&= \chi (\psi (v \cdot \nabla_x M_t)- \psi (v \cdot \nabla_x M_\infty)) \left(\int_{\mathcal{V}} \phi(x,v') \d v' - \phi (x,v)\right) \\
			&\leq 4 \eta \chi V_0\|\psi'\|_\infty \|\nabla_x N\|_\infty  e^{-\gamma M_\infty(x)}.
		\end{align*}		
		In the last line of the above inequality, we used the fact that $\gamma$ is chosen so that $\phi \leq 2 e^{-\gamma M_\infty(x)}$. This gives \eqref{ineq:L_*-L_M}.
	\end{proof}
	\begin{lem} \label{lem:moment_weakly_non-linear}
		Let $f$ be the solution of  Equation \eqref{eq:rt-linear} with  the coupling \eqref{eq:weakly_nonlin}.
		If $\eta$ is sufficiently small, then there exists a constant $B>0$ (not depending on $\eta$) such that 
		\begin{align} \label{est:moment_nonlin}
			\int \phi (z)  F(t,z) \d z\leq \frac{A}{B} + e^{-B t}  \int \phi (z)  F_0(z) \d z,
		\end{align} where $A$ is given by \eqref{A} in the proof of Lemma \ref{lem:Foster-Lyapunov} and $\phi$ is given in \eqref{phi}. In fact we have the bound
		\begin{align} \label{bd:f}
			\|F\|_{*} \leq \frac{A}{B} + \|F_0\|_{*}.
		\end{align} Using equivalence of norms we also have
		\begin{align*}
			\|F\|_{**} \leq C^* + 4 \|F_0\|_{**},
		\end{align*} for $C^*>0$ a constant.
	\end{lem}
	\begin{proof}
		From Lemma \ref{lem:Foster-Lyapunov}, inequality \eqref{ineq:L_*_phi} we know that 
		\begin{align*}
			\mathcal{L}_{M_\infty}^* \phi \leq A - \frac{\gamma \tilde{\lambda} \chi (1-\chi) m_*^k}{6 (1+\chi)}e^{-\gamma M_\infty(x)}. 
		\end{align*}
		Using \eqref{ineq:L_*-L_M} in Lemma \ref{lem:perturb1} we obtain
		\begin{align*}
			\mathcal{L}_{M_t}^* \phi \leq A - \left( \frac{\gamma \tilde{\lambda} \chi (1-\chi) m_*^k}{6 (1+\chi)} - 4 \eta \chi V_0 \|\psi'\|_\infty \|\nabla_x N\|_\infty \right) e^{-\gamma M_{\infty}(x)}.
		\end{align*} Therefore, if we take $\eta$ such that
		\begin{align*}
			\eta \leq \frac{\tilde{\lambda}\chi  (1-\chi) m_*^k}{48 \chi (1+\chi) V_0  \|\psi'\|_\infty \|\nabla_x N\|_\infty },
		\end{align*} then we have for some constant $B >0$
		\begin{align*}
			\frac{\d}{\d t} \int \phi (z)  F (t,z) \mathrm{d}z\leq - B\int \phi (z)  F (t, z) \d z +A \int F (t,z)\d z. 
		\end{align*} Therefore, since the mass is preserved, by Gronwall's inequality we obtain \eqref{est:moment_nonlin}. 
		We can also turn this into an exponential decay on
		\[ \int e^{-\gamma M_\infty(x)} F(t,z)\d z. \] This gives the result.
	\end{proof}
	\begin{lem} \label{lem:lastlemma} 
		Suppose that $F_t$ is the solution of Equation \eqref{eq:rt-linear} with the coupling \eqref{eq:weakly_nonlin} and $F_\infty$ its steady state solution. Suppose that $\eta$ is small enough so that Lemmas \ref{lem:perturb1} and \ref{lem:moment_weakly_non-linear} are valid. Suppose also that
		\begin{align} \label{bd:eta}
			\|F_0\|_{**}  < \frac{1}{4} \left(\frac{\sigma^2}{4\eta \chi  V_0 D \|\psi'\|_\infty  \|\nabla_x N\|_\infty }-C^*\right),
		\end{align} where $\sigma, D$ and $C^*$ are found in Theorem \ref{thm1}, Proposition \ref{prop:unifeta} and Lemma \ref{lem:moment_weakly_non-linear} respectively. Then we have for some $C>0$ that
		\begin{align*}
			\| F_t - F_\infty\|_{**} \leq Ce^{-\sigma t/2} \|F_0 -F_\infty\|_{**}.
		\end{align*}
	\end{lem}
	\begin{proof} We rewrite the weakly non-linear equation \eqref{eq:rt-linear}-\eqref{eq:weakly_nonlin} as
		\begin{align*}
			\p_t F(t,x,v) = \mathcal{L}_{M_t} F(t,x,v) = \mathcal{L}_{\tilde{M}}  F(t,x,v) - (\mathcal{L}_{\tilde{M}} - \mathcal{L}_{M_t} )F(t,x,v), 
		\end{align*} where $\tilde{M}$ is the fixed point of $G$ we found in Proposition \ref{prop:ss_nonlin}. 
	
	Let us call the last term $h = h(t,x,v): = (\mathcal{L}_{\tilde{M}} - \mathcal{L}_{M_t} )F$. Then by Duhamel's formula we have
		\begin{align} \label{eq:nonlin_soln}
			F_t = F(t,x,v) = \mathcal{S}_t ^{\tilde{M}}F_0(x,v) + \int_{0}^{t}  \mathcal{S}_{t- s}^{\tilde{M}} h(s,x,v) \d s.
		\end{align} where $(\mathcal{S}_t^{\tilde{M}})_{t \geq 0}$ is the semigroup associated to Equation \eqref{eq:D2}. Using definitions \eqref{defn:L_M_t} and \eqref{defn:L_M_inf} we have 
		\begin{align*}
			h(t,x,v)= \chi \left( \int_{\mathcal{V}} \left(\psi(v' \cdot \nabla_x M_t)-\psi'(v' \cdot \nabla_x \tilde{M})\right) F(t,x,v')\d v' - \left( \psi(v \cdot \nabla_x M_t) - \psi(v \cdot \nabla_x \tilde{M}) \right)F \right). 
		\end{align*}
		Then, using \eqref{bds:grad_M-grad_M_inf} and \eqref{est:rho_f} from Propositions \ref{prop:unifeta} and \ref{prop:ss_nonlin} respectively, we have
		\begin{align*} 
			\|h\|_{**} &\leq  2\chi \|\psi'\|_\infty V_0 \|\nabla_x M_t -\nabla_x \tilde{M}\|_\infty \|F\|_{**} \\
			& \leq 2\chi \|\psi'\|_\infty V_0 \|\nabla_x \log \left( 1+ \eta N * \rho\right) - \nabla_x \log \left(1+ \eta N*\rho_\infty\right)\|_\infty \|F\|_{**} \\
			& \leq 2 \chi \eta V_0  \|\psi'\|_\infty \|\nabla_x N\|_\infty   \|F_t - F_\infty\|_{**} \|F \|_{**}.
		\end{align*}
		Therefore we obtain 
		\begin{align} \label{bd:h}
			\|h\|_{**} \leq C \eta  \|F_t - F_\infty\|_{**}\|F_t \|_{**}
		\end{align} where $C$ is a constant depending on $\chi,\psi, V_0, N$. Now we subtract $F_{\infty}$ from both sides of \eqref{eq:nonlin_soln} and take the norms to get
		\begin{align} \label{eq:nonlin_norm}
			\|F_t - F_\infty\|_{**}   = 	\| \mathcal{S}_t^{\tilde{M}} F_0 - F_\infty  \|_{**}  + \left 	\|  \int_0^t \mathcal{S}_{t-s}^{\tilde{M}} h(s) \d s\right \|_{**}.
		\end{align}
		We can bound the first term in the right hand side of  \eqref{eq:nonlin_norm} by the result of Theorem \ref{thm1} and the second term by \eqref{bd:h}. Therefore we obtain
		\begin{align*}
			\|F_t - F_\infty\|_{**} \leq C_1e^{-\sigma t}\|F_0 - F_\infty\|_{**} + C\eta \int_0^t e^{-\sigma (t-s)}\|F_s - F_\infty\|_{**} \|F_t\|_{**}\d s,
		\end{align*} where $C>0$, the constant in \eqref{bd:h}, depends on $\chi, \psi, V_0, N$.  By the constraint \eqref{bd:eta} on $\eta$,
		and the bound on $\|F_t\|_{**}$ from Lemma \ref{lem:moment_weakly_non-linear} we have
		\begin{align*}
			\|F_t - F_\infty\|_{**} \leq C_1e^{-\sigma t}\|F_0 - F_\infty\|_{**} + \frac{\sigma}{2} \int_0^t e^{-\sigma (t-s)}\|F_s - F_\infty\|_{**}\d s.
		\end{align*} By Gronwall's inequality this leads to 
		\begin{align*}
			\|F_t - F_\infty\|_{**} \leq C e^{-\sigma t /2} \|F_0 - F_\infty\|_{**} 
		\end{align*} for some constant $C>0$.
		This finishes the proof.
	\end{proof}
	\begin{proof}[Proof of Theorem \ref{thm2}]
		Proposition \ref{prop:ss_nonlin} gives a unique steady state solution for the weakly non-linear equation \eqref{eq:rt-linear}-\eqref{eq:weakly_nonlin}. The exponential relaxation to the steady state solution follows from Lemma \ref{lem:lastlemma}. This completes the proof.
	\end{proof}

	\section{Discussion and future research} \label{sec:discussion}
	
	\subsection{Existence of steady states for fully non-linear models}	
	In this section we discuss the relationship of our work to the much more challenging problem of finding steady states to the run and tumble equation with the fully nonlinear coupling of the form
	\[ -\Delta S + S = \rho.\]  
	Our goal is to describe hopeful direction for future research as well as giving an idea of why we consider the weakly non-linear coupling studied here as a possible stepping stone towards this more complex model. In this regard, we believe that a Schauder fixed point argument is a plausible strategy for finding a steady state of the fully non-linear coupling. We suggest looking for fixed points of the following function $\tilde{G}(M) = \log S$ where $S$ is the solution to 
	\[ -\Delta S + S = \rho^M, \] where $\rho^M$ is the spatial marginal of the unique steady state of \eqref{eq:rt-linear} with the log-chemoattractant $M$.
	
	The first step is to determine if the estimates we obtain in Section \ref{sec:Foster-Lyapunov} (the Foster-Lyapunov part) would be good enough to run such a fixed point argument, that is, we would like to see if the bounds we find on 
	\[ \int f^M \phi \d z, \] are sufficient to find a compact, convex set of possible chemoattractant densitites which is preserved by $\tilde{G}$. Since there is a one-to-one correspondence between $\tilde{G}(M)$ and $\rho^M$, this is equivalent to finding a set of possible $\rho^M$. A standard way of showing the necessary compactness would be to show tightness of the measures $\rho^M$, and this can be achieved by proving moment estimates (such as are found in the Foster-Lyapunov part). However, we encounter a problem that at each iteration of such a scheme, we lose weight in our moment estimate.

	In this paper, we experiment with a toy non-linear model, where we could use the estimates coming from the Foster-Lyapunov part to be able to use a fixed point argument. This gives us a better understanding of how this type of argument should work. We briefly describe our process for choosing this coupling.
	
	The first idea was to come up with a perturbative setting to try a coupling of the form
	\begin{align}
		\label{coupling1}
		-\Delta S + S = \rho_* + \eta \rho, 
	\end{align} where $\rho_* $ is a fixed spatial density and $\eta$ is a small number. However, we notice that this coupling has essentially exactly the same problem with a loss of weight as the fully non-linear coupling. In order to create a coupling we can deal with, the $\eta \rho $ in the right hand side of \eqref{coupling1} needs to be multiplied by a function of $x$ that decays sufficiently fast at infinity. Therefore, we can try a coupling that looks like 
	\begin{align}	\label{coupling2}
		-\Delta S + S = \rho_*  (1+\eta\rho).
	\end{align} 
	Then, $S$, which is the solution of  \eqref{coupling2}, is given by 
	\begin{align} \label{coupling3}
		S = N*(\rho_*(1+\eta \rho)),
	\end{align} where $N =  \mathcal{F}^{-1} (1/(1+ |\xi|^2))$, and $\mathcal{F}$ represents the Fourier transform. Then, we further simplify \eqref{coupling3} as 
	\begin{align*}
		S= S_\infty(1+\eta N*\rho)
	\end{align*} where $N$ is now a positive, smooth function and $S_{\infty}$ is a smooth function. Considering this simplification allows us to keep algebra simple without losing the behaviour of \eqref{coupling3}. By this strategy we obtain the weakly non-linear, nonlocal coupling introduced in \eqref{eq:weakly_nonlin}. Even though this weakly non-linear coupling serves as a toy model we still retain the idea of a fixed point argument on the chemoattractant profile.
	
	Our contraction mapping argument is an adaption of what was originally an argument to show continuity of a map $\tilde{G}$ defined on a fully non-linear coupling. In order to carry out a Schauder fixed point argument, continuity of such $\tilde{G}$ would be needed.
	
	Finally, the toy model we introduced, even though biologically not realistic, allows us to understand better how to use the arguments presented in this paper in the fully non-linear setting. This is a subject of ongoing work. 
	
	\subsection{Physically more realistic tumbling kernels}
	The methods used here should be able to accomodate more complex tumbling kernels. In particular we would like to look at models with unbounded velocity spaces (where we would expect to see polynomial rather than exponential rates of convergence) and tumbling kernels where the bacteria can only turn by a bounded angle. 
	
	The experiments conducted in \cite{BB72} shows that, for peritrichous bacteria such as \emph{E. coli}, the tumbling kernel $\kappa$ depends only on the relative angle $\theta$ between the pre- and post-tumbling velocities $v$ and $v'$ respectively. Particularly, for bacteria \emph{E. coli}, the tumbling kernel $\kappa$ is given by 
	\begin{align} \label{eq:K_coli}
		\kappa (v, v') = \frac{g(\theta)}{2 \pi \sin \theta} \quad \text{where} \quad \theta = \arccos \left(\frac{v \cdot v'}{|v||v'|}\right), 
	\end{align} where $g(\theta)$ is the sixth order polynomial satisfying $g(0) = g(\pi) =0$ (see \cite{OH02, BFSC96}).

	They also suggest the following form of tumbling rate
	\[ \lambda = \lambda_0 \exp \left( - \frac{c_1 k_D}{(k_D +S)^2} v\cdot \nabla_x S \right),\] where $\lambda_0, c_1$ and $k_D$ are constants and $S$ is the chemoattractant density. In an ongoing work we study these more realistic versions of the run and tumble model.

	\section*{Acknowledgements}
	The authors wish to thank E. Bouin, J. A. Cañizo, V. Calvez, and G. Raoul for useful comments and discussions. J. Evans was supported by FSPM postdoctoral fellowship and the grant ANR-17-CE40-0030 (both until July 2020).
	H. Yolda\c{s} was supported by the European Research Council (ERC) under the European Union’s Horizon 2020 research and innovation programme (grant agreement No  865711). The authors gratefully acknowledge the support of the Hausdorff Institute for Mathematics (Bonn) through Junior Trimester Program on Kinetic Theory.

\appendix
\section{Cauchy theory for the weakly non-linear equation}
For $T>0$ fixed given $f_t \in L^\infty_t ( [0,T];\mathcal{P}_{x,v})$ we define the function 
\[ M^f (t,x) = M_\infty + \log \left ( 1+ \eta \int N(x-y) f_t(y, v) \d y \d v \right ).  \] Then, if we fix initial data $f_0 \in \mathcal{P}_{x,v}$ we can define a function from $L^\infty_t ( [0,T];\mathcal{P}_{x,v})$ to itself via

%\begin{align*} 
%\mathscr{H}(f)_t &= \exp \left( -\int_0^t \lambda(v \cdot \nabla_x M^f (s,x-vs) ) \mathrm{d}s \right) f_0(x-vt, v)   \\
%& \quad + \int_0^t \exp \left( - \int_s^t \lambda( v %\nabla_x M^f (r, x-vr)) \mathrm{d}r \right) %\int_{\mathcal{V}} \lambda(v' \cdot \nabla_x %M^f(s,x-v(t-s))) f_s (x, v') \mathrm{d}v' \mathrm{d}s.
%\end{align*}

\begin{multline*} 
	\mathscr{H}(f)_t = e^{-\int_0^t \lambda(v \cdot \nabla_x M^f (s,x-vs) ) \mathrm{d}s} f_0(x-vt, v)   \\
	 \quad + \int_0^t e^{- \int_s^t \lambda( v \nabla_x M^f (r, x-vr)) \d r }\int_{\mathcal{V}} \lambda(v' \cdot \nabla_x M^f(s,x-v(t-s))) f_s (x, v') \mathrm{d}v' \mathrm{d}s.
\end{multline*}

\begin{lem} \label{lem:contraction}
Suppose $\psi$ is uniformly Lipschitz continuous then there exists a $T_*>0$ depending only on $\psi, \chi$ such that if $T \leq T_*$ then $\mathscr{H}$ is a contraction.
\end{lem}
\begin{proof}
We have the following computations
%\begin{align*} \sup_{t \leq T}& \|\exp \left( -\int_0^t \lambda(v \cdot \nabla_x M^{f^1} (s,x-vs) ) \mathrm{d}s \right) f_0(x-vt, v)  - \exp \left( -\int_0^t \lambda(v \cdot \nabla_x M^{f^2} (s,x-vs) ) \mathrm{d}s \right) f_0(x-vt, v) \|_{TV}   \\
%& \leq \sup_{t,x,v}| \exp \left( -\int_0^t \lambda(v \cdot \nabla_x M^{f^1} (s,x-vs) ) \mathrm{d}s \right) - \exp \left( -\int_0^t \lambda(v \cdot \nabla_x M^{f^2} (s,x-vs) ) \mathrm{d}s \right)  | \\
%& \leq T \sup_{t,x,v} | \lambda(v \cdot \nabla_x M^{f^1}(t,x-vt)) - \lambda(v \cdot \nabla_x M^{f^2}(t, x-vt))| \\
%& \leq T \eta V_0 \|\psi\|_{Lip} |\nabla_x M^{f^1}(t,x-vt) - \nabla_x M^{f^2}(t,x-vt)| \\
%& \leq T \eta V_0 \|\psi\|_{Lip} \|\nabla_x N\|_\infty \sup_{t \leq T} \|f^1_t-f^2_t\|_{TV}.
%\end{align*}
	\begin{align*} \sup_{t \leq T} \Big \| &e^{-\int_0^t \lambda(v \cdot \nabla_x M^{f^1} (s,x-vs) ) \d s } f_0(x-vt, v)  - e^{-\int_0^t \lambda(v \cdot \nabla_x M^{f^2} (s,x-vs) ) \d s } f_0(x-vt, v) \Big \|_{TV}   \\
		& \leq \sup_{t,x,v} \Big | e^{-\int_0^t \lambda(v \cdot \nabla_x M^{f^1} (s,x-vs) ) \d s } - e^{ -\int_0^t \lambda(v \cdot \nabla_x M^{f^2} (s,x-vs) ) \d s } \Big   | \\
		& \leq T \sup_{t,x,v} \Big | \lambda(v \cdot \nabla_x M^{f^1}(t,x-vt)) - \lambda(v \cdot \nabla_x M^{f^2}(t, x-vt)) \Big | \\
		& \leq T \eta V_0 \|\psi\|_{Lip} \left |\nabla_x M^{f^1}(t,x-vt) - \nabla_x M^{f^2}(t,x-vt) \right | \\
		& \leq T \eta V_0 \|\psi\|_{Lip} \|\nabla_x N\|_\infty \sup_{t \leq T} \|f^1_t-f^2_t\|_{TV}.
\end{align*}
We also have
%\begin{align*}
%&\| \int_{\mathcal{V}} \lambda(v' \cdot \nabla_x M^{f^1} (s, x-v(t-s)))f^1_s(x,v')\mathrm{d}v' - \int_{\mathcal{V}} \lambda(v' \cdot \nabla_x M^{f^2} (s, x-v(t-s)))f^2_s(x,v')\mathrm{d}v'\|_{TV}\\
%& \leq |B(V_0)|\| \lambda(v' \cdot \nabla_x M^{f_1} )- \lambda(v' \cdot \nabla_x M^{f_2}) \|_\infty + (1+\chi)|B(V_0)| \| f_t^1- f_t^2\|_{TV}\\
%& \leq \C(V_0,\eta, \|\psi\|_{Lip}, \chi) \sup_{t \leq T} \|f^1_t - f^2_t\|_{TV}.
%\end{align*}
\begin{align*}
	 \Big \| \int_{\mathcal{V}} \lambda(v' \cdot \nabla_x M^{f^1} &(s, x-v(t-s))) f^1_s(x,v')\mathrm{d}v' - \int_{\mathcal{V}} \lambda(v' \cdot \nabla_x M^{f^2} (s, x-v(t-s)))f^2_s(x,v')\mathrm{d}v' \Big \|_{TV}\\
	 &\leq |B(V_0)|\| \lambda(v' \cdot \nabla_x M^{f_1} )- \lambda(v' \cdot \nabla_x M^{f_2}) \|_\infty + (1+\chi)|B(V_0)| \| f_t^1- f_t^2\|_{TV}\\
	 &\leq C (V_0,\eta, \|\psi\|_{Lip}, \chi) \sup_{t \leq T} \|f^1_t - f^2_t\|_{TV}.
\end{align*}
which implies working as in the first computation
%\begin{align*}
%&\sup_{t \leq T}\| \int_0^t \exp \left( - \int_s^t \lambda( v \nabla_x M^{f^1} (r, x-vr)) \mathrm{d}r \right) \int_{\mathcal{V}} \lambda(v' \cdot \nabla_x M^{f^1}(s,x-v(t-s))) f^1_s (x, v') \mathrm{d}v' \mathrm{d}s \\
%&\int_0^t \exp \left( - \int_s^t \lambda( v \nabla_x M^{f^2} (r, x-vr)) \mathrm{d}r \right) \int_{\mathcal{V}} \lambda(v' \cdot \nabla_x M^{f^2}(s,x-v(t-s))) f^2_s (x, v') \mathrm{d}v' \mathrm{d}s\|_{TV} \\
%&\leq TC(V_0, \|\psi\|_{Lip}, \eta, \chi) \sup_{t \leq T} \|f^1_t - f^2_t\|_{TV}.
%\end{align*}
	\begin{multline*}
		\sup_{t \leq T} \Big \| \int_0^t e^{- \int_s^t \lambda( v \nabla_x M^{f^1} (r, x-vr)) \d r } \int_{\mathcal{V}} \lambda(v' \cdot \nabla_x M^{f^1}(s,x-v(t-s))) f^1_s (x, v') \d v' \d s \\
		 - \int_0^t e^{ - \int_s^t \lambda( v \nabla_x M^{f^2} (r, x-vr)) \d r }\int_{\mathcal{V}} \lambda(v' \cdot \nabla_x M^{f^2}(s,x-v(t-s))) f^2_s (x, v') \d v' \d s \Big \|_{TV} \\
		\leq TC(V_0, \|\psi\|_{Lip}, \eta, \chi) \sup_{t \leq T} \|f^1_t - f^2_t\|_{TV}.
\end{multline*}
Then putting this together gives
\[ \sup_{t \leq T} \|\mathscr{H}(f^1)_t - \mathscr{H}(f^2)_t\|_{TV} \leq T C(V_0, \|\psi\|_{Lip}, \eta, \chi) \sup_{t \leq T} \|f^1_t -f^2_t\|_{TV}. \] Therefore if $T$ is small enough $\mathscr{H}$ is a contraction. 
\end{proof}

\begin{prp}
Equation \eqref{eq:rt-linear} with the coupling \eqref{eq:weakly_nonlin} has a unique global solution in $L^\infty([0, \infty); \mathcal{P}_{x,v})$.
\end{prp}
\begin{proof}
On the time interval $[0,T]$ with $T \leq T_*$ from Lemma \ref{lem:contraction} above we have a unique fixed point of the map $\mathscr{H}$ which gives us a solution of \eqref{eq:rt-linear} with the coupling \eqref{eq:weakly_nonlin} on this time interval. Then since $T_*$ does not depend on $f_0$, we can iterate this forwards in time to build solutions on arbitrarily long time intervals.
\end{proof}
	
	\bibliography{Run-and-Tumble}

\end{document}